\documentclass{amsart}
\usepackage{amssymb}
\usepackage{amsmath}
\usepackage{color}
\usepackage{graphicx}
\def\E{{\mathbb E}}
\def\P{{\mathbb P}}
\def\R{{\mathbb R}}

\def\N{{\mathbb N}}

\def\ge {{\varepsilon}}

\newtheorem{theo}{Theorem}
\newtheorem{prop}{\indent Proposition}
\newtheorem{lem}{\indent Lemma}

\newtheorem{rem}{\indent Remark}
\newtheorem{ass}{Assumption}
\newtheorem{cor}{\indent Corollary}

\title[Interacting neurons : estimation of the spiking rate]{Non-parametric estimation of the spiking rate in systems of interacting neurons.}
\date{September 22, 2016}
\author{P.~Hodara, N. Krell, E. L\"ocherbach }

\address{P. Hodara:  CNRS UMR 8088, D\'epartement de Math\'ematiques, Universit\'e de Cergy-Pontoise,
2 avenue Adolphe Chauvin, 95302 Cergy-Pontoise Cedex, France.}

\email{pierre.hodara@sfr.fr}

\address{N. Krell: Universit\'e de Rennes 1,
 Institut de Recherche math\'ematique de Rennes,
CNRS-UMR 6625, Campus de Beaulieu.
 B\^atiment 22, 35042 Rennes Cedex, France.}
\email{nathalie.krell@univ-rennes1.fr}

\address{E. L\"ocherbach: CNRS UMR 8088, D\'epartement de Math\'ematiques, Universit\'e de Cergy-Pontoise,
2 avenue Adolphe Chauvin, 95302 Cergy-Pontoise Cedex, France.}

\email{eva.loecherbach@u-cergy.fr}



\begin{document}

\maketitle

\begin{abstract}

We consider a model of interacting neurons where the membrane potentials of the neurons are described by a multidimensional piecewise deterministic Markov process (PDMP) with values in $\R^N, $ where $ N$ is the number of neurons in the network. A deterministic drift attracts each neuron's membrane potential to an equilibrium potential $m.$ When a neuron jumps, its membrane potential is reset to a resting potential, here $0,$ while the other neurons receive an additional amount of potential $\frac{1}{N}.$ We are interested in the estimation of the jump (or spiking) rate of a single neuron based on an observation of the membrane potentials of the $N$ neurons up to time $t.$ We study a Nadaraya-Watson type kernel estimator for the jump rate and establish its rate of convergence in $L^2 .$ This rate of convergence is shown to be optimal for a given H\"older class of jump rate functions. We also obtain a central limit theorem for the error of estimation. The main probabilistic tools are the uniform ergodicity of the process and a fine study of the invariant measure of a single neuron. 
\end{abstract}

 \section{Introduction}
This paper is devoted to the statistical study of certain Piecewise Deterministic Markov Processes (PDMP) modeling the activity of a biological neural network. More precisely, we are interested in estimating the the underlying jump rate of the process, \textit{i.e.}\ the spiking rate function of each single neuron. 

Piecewise Deterministic Markov Processes (PDMP's) have been introduced by Davis (\cite{Davis84} and \cite{Davis93}) as a family of c\`adl\`ag Markov processes following a deterministic drift with random jumps. PDMP's are widely used in probabilistic modeling of \textit{e.g.}\ biological or chemical phenomena (see \textit{e.g.} \cite{CDMR} or \cite{PTW-10}, see \cite{ABGKZ} for an overview). In the present paper, we study the particular case of PDMP's which are systems of interacting neurons. Building a model for the activity of a neural network that can fit biological considerations is crucial in order to understand the mechanics of the brain. Many papers in the literature use Hawkes Processes in order to describe the spatio-temporal dependencies which are typical for huge systems of interacting neurons, see \cite{ae}, \cite{hrbr} and \cite{HL} for example. Our model can be interpreted as Hawkes process with memory of variable length (see \cite{aenew});  it is close to the model presented in \cite{bresiliens}. It is of crucial interest for modern neuro-mathematics to be able to statistically identify the basic parameters defining the dynamics of a model for neural networks. The most relevant mechanisms to study are the way the neurons are connected to each other and the way that a neuron deals with the information it receives. In \cite{dglo} and in \cite{hrbr}, the authors build an estimator for the interaction graph, in discrete or in continuous time. In the present work, we assume that we observe a subsystem of neurons which are all interconnected and behaving in a similar way. We then focus on the estimation of the firing rate of a neuron within this system. This rate depends on the membrane potential of the neuron, influenced by the activity of the other neurons. 

More precisely, we consider a process $X_t=(X_t^1,...,X_t^N),$ where $N$ is the number of neurons in the network and where each variable $X_t^i$ represents the membrane potential of neuron $ i, $ for $ 1 \le i \le N.$ Each membrane potential $X_t^i $ takes values in a compact interval $ [0, K], $ where $0$ is interpreted as resting potential (corresponding to $ \sim - 90 mV$ in real neurons) and where $ K \sim 140 m V $ (see \textit{e.g.} \cite{I-09}). This process has the following dynamic. A deterministic drift attracts the membrane potential of each neuron to an equilibrium potential $m \in \R_+$ with an exponential speed of parameter $\lambda \in \R_+ .$ Moreover, a neuron with membrane potential $x$ ``fires" (\textit{i.e.}, jumps) with intensity $f(x),$ where $f:\R_+ \to \R_+ $ is a given intensity function. When a neuron fires, its membrane potential is reset to $0,$ interpreted as resting potential, while the membrane potentials of the other neurons are increased by $\frac{1}{N}$ until they reach the maximal potential height $K.$

The goal of this paper is to explore the statistical complexity of the model described above in a non-parametric setting. We aim at giving precise statistical characteristics (such as optimal rates of
convergence, estimation procedures) such that we are able to compare systems of interacting neurons
to benchmark non-parametric models like density estimation or nonlinear
regression. More precisely, given the continuous observation \footnote{A short remark concerning the continuous time observation scheme : Presumably, if we deal with discrete time samples, observed at sufficiently high frequency such that with huge probability at most one jump can take place during one sampling step, it would be possible to reconstruct the continuous trajectory of the process with hight probability and to perform our estimation procedure also in this frame. } of the system of interacting neurons over a time interval $[0,t]$ (with asymptotics being taken as $t \rightarrow \infty$), we infer on the different parameters of the model
which are: the equilibrium potential $m,  $ the speed of attraction $\lambda $ and the spiking rate function $f$. Since in a continuous time setting, the coefficients $ \lambda $ and  $m$ are known (they can be identified by any observation of the continuous trajectory of a neuron's potential between two successive jumps), the {\it typical} problem is the
estimation of the unknown spiking rate $f (\cdot).$ 

Therefore we restrict our attention to the estimation of the unknown spiking rate $f (\cdot) .$ We measure smoothness of the spiking rate by considering H\"older classes of possible shapes for the spiking rate and suppose that the spiking rate has smoothness of order $\beta $ in a H\"older sense. To estimate the jump rate $f$ in a position $a,$ we propose a Nadaraya-Watson type kernel estimator which is roughly speaking of the form
$$
\hat f_t(a)=\frac{\sharp \mbox{ spikes in positions in $ B_h(a)$ during $[0, t ] $} }{\mbox{ occupation time of $ B_h(a)$ during $[0, t ]$ } },
$$
where $B_h(a)$ is a neighborhood of size $h$ of the position $a$ where we estimate the jump rate function $f.$ A rigorous definition of the estimator is given in terms of the jump measure and an occupation time measure of the process $X.$ The convergence of the estimator is implied by the fact that the compensator of the jump measure is the occupation time measure integrated against the jump rate function $f,$ together with uniform ergodicity of the process. Assuming that the jump rate function $f$ has smoothness of order $\beta $ in a H\"older sense, we obtain the classical rate of convergence of order $t^{-\frac{\beta}{2 \beta +1}}$ for the point-wise $L^2 -$error of the estimator. This rate is shown to be optimal. We also state two important probabilistic tools that are needed in order to obtain the statistical results. The first one is the uniform positive Harris recurrence of process. The second one is the existence of a regular density function of the invariant measure of a single neuron. 

In the literature, non-parametric estimation for PDMP's has already been studied,  see for example \cite{ADG-P} and, more particularly concerning the estimation of the jump rate, \cite{AM-G}. On the contrary to these studies, the framework of the present work is more difficult for two reasons.  The first reason is the fact that our process is multidimensional, presenting real interactions between the neurons. Of course, estimation problems for multidimensional PDMP's have already been studied. However, in all cases we are aware of, a so-called ``many-to-one formula"  (see \cite{Nathalie}, see also \cite{hoffmann-olivier}) allows to express the occupation time measure of the whole system in terms of a single ``typical'' particle. This is not the case in the present paper -- and it is for this reason that we have to work under the relatively strong condition of uniform ergodicity which is implied by compact state space -- a condition which is biologically meaningful. The second, more important, reason is the fact that the transition kernel associated to jumps is degenerate. This is why the construction of our estimator is different from other constructions in previous studies. The degeneracy of the transition kernel also leads to real difficulties in the study of the regularity of the invariant density of a single neuron, see \cite{evanew} and the discussions therein.

In Section \ref{sec:results}, we describe more precisely our model and state our main results. We first provide two probabilistic results necessary to prove the convergence of the estimator: firstly, the positive Harris recurrence of the process $X$ in Theorem \ref{theo:harrisok} and secondly the properties of the invariant measure in Theorem \ref{theo:invmeasure}. The speed of convergence of our estimator is established in Theorem \ref{theo:main}. Finally, Theorem \ref{theo:lowerbound} states that our speed of convergence is optimal for the point-wise $L^2-$error, uniformly in $f.$ The key tool to prove this optimality is to study the asymptotic properties of the likelihood process for a small perturbation of the function $f$ close to $a.$

The proofs of Theorems \ref{theo:harrisok},\ref{theo:main} and \ref{theo:lowerbound} are respectively given in Sections \ref{sec:Harris}, \ref{sec:proofmain} and \ref{sec:optimal}. We refer the reader to \cite{evanew} for a proof of Theorem \ref{theo:invmeasure}.

\section{The model}\label{sec:results}

\subsection{The dynamics}
Let $N > 1 $ be fixed and $(N^i(ds, dz))_{i=1,\dots,N}$ be a family of \textit{i.i.d.}\ Poisson random measures on $\R_+ \times \R_+ $ having intensity measure $ds dz.$ We study the Markov process $X_t = (X^{ 1 }_t, \ldots , X^{ N}_t )$
taking values in $[0, K]^N$  and solving, for $i=1,\dots,N$, for $t\geq 0$,
\begin{eqnarray}\label{eq:dyn}
X^{ i}_t &= & X^{i}_0  - \lambda \int_0^t ( X^{ i}_s - m) ds -  \int_0^t \int_0^\infty 
X^{ i}_{s-}  1_{ \{ z \le  f ( X^{ i}_{s-}) \}} N^i (ds, dz) \\
&&+    \sum_{ j \neq i } \int_0^t\int_0^\infty  a_K( X_{s-}^i ) 1_{ \{ z \le  f ( X^{j}_{s-}) \}} N^j (ds, dz).
\nonumber
\end{eqnarray}  
In the above equation, $\lambda > 0 $ is a positive number, $ m$ is the equilibrium potential value such that $0 < m < K.$ Moreover, we will always assume that $ K \geq \frac2N .$ Finally, the functions $a_K : [0, K ] \to [0, K ]  $  and $f:\R_+\mapsto \R_+$  
satisfy (at least) the following assumption.

\begin{ass}\label{ass:1}~\\
1. $a_K : [0, K ] \to [0, \frac{1}{N}]  $ is non-increasing and smooth, $ a_K ( x) = \frac1N, $ for all $x < K- \frac2N$ and $ a_K ( x) < K-x$ for all $x \geq K- \frac2N .$ \\
2. $f\in C^1(\mathbb R_+),$ $f$ is non-decreasing, $f(0) = 0,$ and there exists $f_{min}:\R_+\mapsto \R_+,$ non-decreasing, such that $f(x) \geq f_{min}(x) > 0 $ for all $x > 0$.
\end{ass}
All membrane potentials take values in $ [0, K ], $ where $K$ is the maximal height of the membrane potential of a single neuron. $0$ is interpreted as resting potential (corresponding to $ \sim - 90 mV$ in real neurons) and $ K \sim 140 m V $ (see \textit{e.g.} \cite{I-09}). In \eqref{eq:dyn}, $\lambda $ gives the speed of attraction of the potential value of each single neuron to an equilibrium value $m .$ The function $a_K$ denotes the increment of membrane potential received by a neuron when an other neuron fires. For neurons with membrane potential away from the bound $K,$ this increment is equal to $\frac{1}{N}.$ However, for neurons with membrane potential close to $K,$ this increment may bring their membrane potential above the bound $K.$ This is why we impose this dynamic close to the bound $K.$

In what follows, we are interested in the estimation of the intensity function $f,$ assuming that the parameters $K, f_{min}$ and $a_K$ are known and that the function $f$ belongs to a certain H\"older class of functions. The parameters of this class of functions are also supposed to be known. The assumption $f(0)=0$ comes from biological considerations and expresses the fact that a neuron, once it has fired, has a refractory period during which it is not likely to fire. 

The generator of the process $X$ is given for any smooth test function $ \varphi : [0,K]^N \to \R $  and $x \in [0, K]^N$ by
\begin{equation}\label{eq:generator0}
L \varphi (x ) = \sum_{ i = 1 }^N f(x_i) \left[ \varphi ( \Delta_i ( x)  ) - \varphi (x) \right]
- \lambda \sum_i \left( \frac{\partial \varphi}{\partial x_i} (x) \left[  x_i -m  \right]  \right) ,
\end{equation}
where
\begin{equation}\label{eq:delta}
(\Delta_i (x))_j =    \left\{
\begin{array}{ll}
x_j + a_K ( x_j )  & j \neq i \\
0 & j = i 
\end{array}
\right\}
 .
\end{equation}

The existence of a process $X$ with such dynamics is ensured by an acceptance/rejection procedure that allows to construct solutions to (\ref{eq:dyn}) explicitly. More precisely, since each neuron spikes at maximal intensity $f(K),$ we can work conditionally on the realization of a Poisson process $\bar N$ with intensity $Nf(K).$ We construct the process $X$ considering the jump times $\bar T_n$  of $\bar N$ as candidates for the jump times of $X$ and accepting them with probability 
$$\frac{\sum_{i=1}^Nf\left( X_{\bar T_n -}^i \right) }{Nf(K)}.$$
It is then possible to construct a solution to \eqref{eq:dyn} step by step, following the deterministic drift between the jump times of $\bar N,$ and jumping according to this acceptance/rejection procedure. We refer the reader to Theorem 9.1 in chapter IV of \cite{IW} for a proof of the existence of the process $(X_t)_t.$

We denote by $P_x$ the probability measure under which the solution $(X_t)_t$ of \eqref{eq:dyn} starts from $X_0 = x \in [0,K]^N . $ Moreover, $P_\nu = \int_{[0, K]^N } \nu  (dx) P_x $ denotes the probability measure under which the process starts from $X_0 \sim \nu.$ Figure \ref{Fig. 1} is an example of trajectory for $N=5$ neurons, choosing $f=Id, \; \lambda=1, \; m=1,$ and $K=2.$

\begin{figure}[!h]
\includegraphics[scale=0.55]{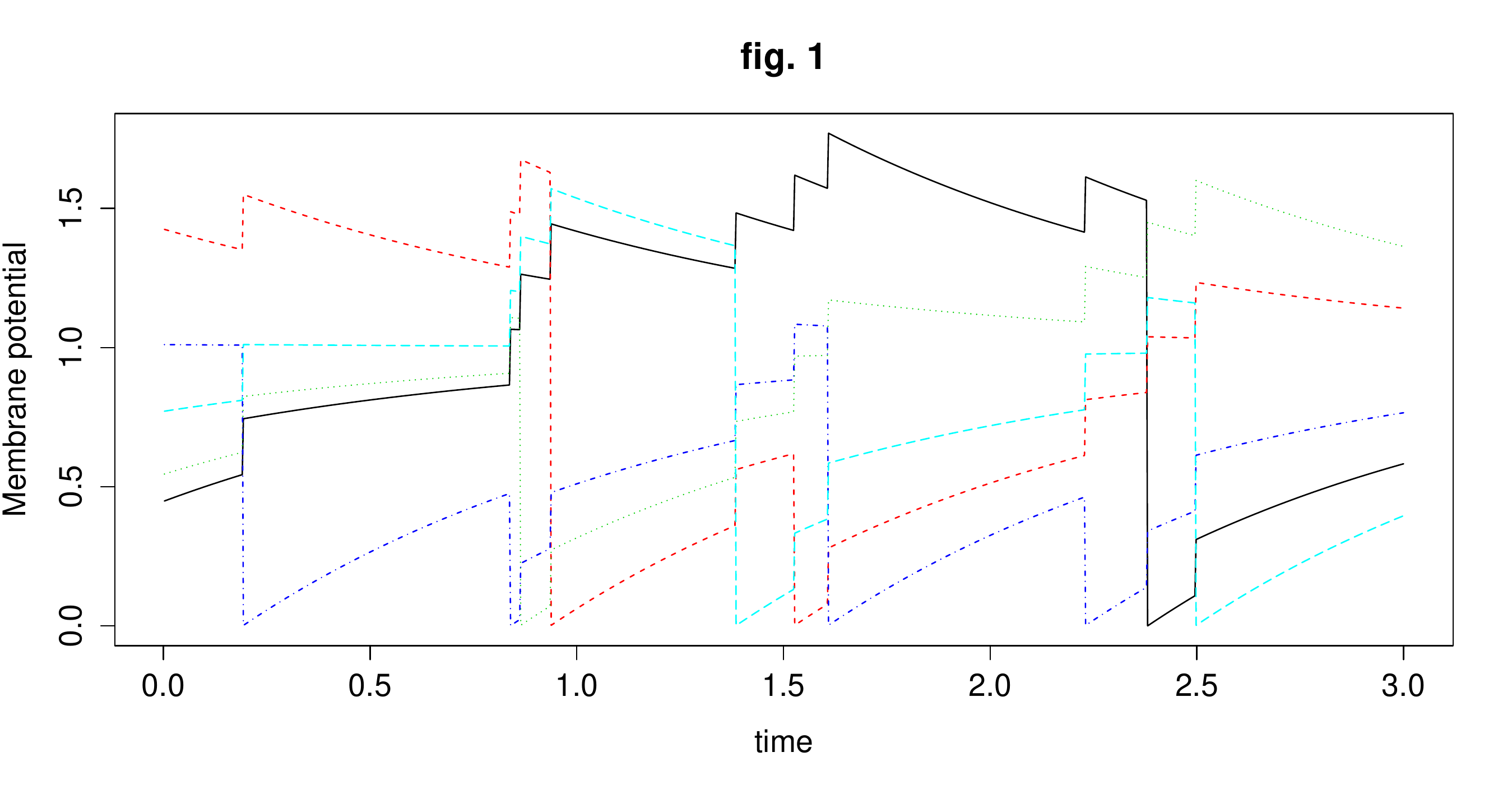} 
\caption{Trajectory of 5 neurons}
\label{Fig. 1}
\end{figure}

The aim of this work is to estimate the unknown firing rate function $f$ based on an observation of $ X$ continuously in time. Notice that for all $ 1 \le i \le N , $ $X^i$ reaches the value $0$ only through jumps. Therefore, the following definition gives the successive spike times of the $i-$th neuron, $1 \le i \le N.$ We put
$$ T^i_0 = 0 , T^i_n = \inf\{ t > T^i_{n-1} : X^i_{ t- } > 0 , X^i_t = 0 \} , n \geq 1 ,$$
and introduce the jump measures 
$$ \mu^i ( ds, dy ) = \sum_{n \geq 1 } 1_{ \{  T^i_n < \infty \} }  \delta_{ (T^i_n, X^i_{T^i_n - }) }  (dt, dy), \quad \mu ( dt, dx) = \sum_{i=1}^N \mu^i ( ds, dx) .$$
By our assumptions, $ \mu^i$ is compensated by $ \hat \mu^i ( ds, dy ) = f( X^i_s) ds \delta_{ X^i_s} ( dy) ,$ and therefore the compensator $\hat \mu $ of $\mu $ is given by 
$$ \hat \mu (dt, dy) =   f(y) \eta(dt,dy) , \mbox{ where } \eta(A \times B) =  \int_A \left( \sum_{i= 1}^N 1_B ( X^i_s)  \right) ds $$
is the total occupation time measure of the process $X.$

We will also write $T_n , n \geq 0, $ for the successive jump times of the process $X,$ \textit{i.e.}\
$$ T_0 = 0 , T_n = \inf\{ T_k^i : T_k^i > T_{ n -1}, k \geq 1 , 1 \le i \le N \}, n \geq 1 .$$

For some kernel function $Q$ such that 
\begin{equation}\label{eq:intkern}
Q \in C_{c} ( \R) , \int_{\R} Q(y) dy = 1 , 
\end{equation}
we define the kernel estimator for the unknown function $f$ at a point $a$ with bandwidth $h$, based on observation of $X$ up to time $t$ by 
\begin{equation}\label{eq:kernelest}
 \hat f_{t,h}  (a) =\frac{\int_0^t \int_{\R} Q_h(y-a) \mu(ds,dy) }{\int_0^t \int_{\R} Q_h(y-a) \eta(ds,dy)},  \mbox{ where } Q_h(y):=\frac{1}{h}Q\left(\frac{y}{h}\right) \mbox{ and } \frac{0}{0} := 0 . 
\end{equation}  
For $h$ small, $ \hat f_{t,h}  (a)$ is a natural estimator for $ f( a) .$ Indeed, this expression as a ratio follows the intuitive idea to count the number of jumps that occurred with a position close to $a$ and to divide by the occupation time of a neighborhood of $a,$ which is natural to estimate an intensity function depending on the position $a.$ More precisely, by the martingale convergence theorem, the numerator $\int_0^t \int_{\R} Q_h(y-a) \mu(ds,dy) $ should behave, for $ t$ large, as $ \int_0^t \int_{\R} Q_h(y-a) f (y) \eta (ds,dy) .$ But by the ergodic theorem, 
$$ \frac{\int_0^t \int_{\R} Q_h(y-a) f (y) \eta (ds,dy)}{\int_0^t \int_{\R} Q_h(y-a) \eta(ds,dy)} \to \frac{\pi_1( Q_h (\cdot - a) f ) }{\pi_1 ( Q_h ( \cdot - a ) ) } $$
as $t \to \infty ,$ where $ \pi_1$ is the stationary measure of each neuron $X_t^i .$ Finally, if the invariant measure $\pi_1$ is sufficiently regular, then 
$$\frac{\pi_1( Q_h (\cdot - a) f ) }{\pi_1 ( Q_h ( \cdot - a ) ) } \to  f(a) $$
as $h \to 0 . $ 

We restrict our study to fixed H\"older classes of rate functions $f.$ For that sake, we introduce the notation $ \beta = k + \alpha $ for $ k=\lfloor 
\beta \rfloor \in \N $ and $ 0 \le \alpha < 1.$ We consider the following H\"older class for arbitrary constants $F, L  > 0,$ and a function $f_{min}$ as in Assumption \ref{ass:1}.
\begin{multline}\label{eq:Hspace}
 H( \beta,F, L ,f_{min}) = \{ f \in C^k ( \R_+) : | \frac{d^l}{d x^l } f(x) | \leq F , \mbox{ for all }  0 \le l \le k, x \in [0, K] , \; \\
  f(x) \geq f_{min}(x) \mbox{ for all $x\in [0, K]$,}  \; \ | f^{(k)} (x) - f^{(k) } (y) | \le L |x- y |^\alpha  \mbox{ for all } x, y \in [0, K]\} .
\end{multline}

\subsection{Probabilistic results}
In this Section, we collect important probabilistic results. We first establish that the process $(X_t)_{t \geq 0} $ is recurrent in the sense of Harris. 

\begin{theo}\label{theo:harrisok}
Grant Assumption \ref{ass:1}. Then the process $X$ is positive Harris recurrent having unique invariant probability measure $\pi ,$ \textit{i.e.}\ for all $ B \in {\mathcal B} ( [0, K ]^N) , $ 
\begin{equation}\label{eq:definrec}
\pi  (B) > 0 \; \mbox{ implies }    P_x \left( \int_0^\infty 1_B (X_s) ds = \infty \right) =1 
\end{equation}
for all $ x \in [0, K]^N  .$  Moreover, there exist constants $C > 0 $ and  $ \kappa > 1 $ which do only depend on the class $H( \beta , F, L, f_{min} ) , $ but not on $f,$ such that 
\begin{equation}\label{eq:ergodic}
\sup_{ f \in H( \beta , F, L, f_{min} ) } \| P_t (x, \cdot ) - \pi \|_{TV} \le C \kappa^{ - t } .
\end{equation}
\end{theo}

It is well-known that the behavior of a kernel estimator such as the one introduced in \eqref{eq:kernelest} depends heavily on the regularity properties of the invariant probability measure of the system. Our system is however very degenerate. Firstly, it is a piecewise deterministic Markov process (PDMP) in dimension $N,$ with interactions between particles. Hence, no Brownian noise is present to smoothen things. Moreover, the transition kernels associated to the jumps of system \eqref{eq:dyn} are highly degenerate (recall \eqref{eq:delta}). The transition kernel 
$$ K ( x, dy) = {\mathcal L} ( X_{T_1} |X_{T_1- } = x ) (dy ) = \sum_{i=1}^N \frac{f ( x^i )}{\bar f (x) } \delta_{ \Delta^i (x) } ( dy ) $$
with $\bar f(x):= \sum_{i=1}^N f(x^i)$ puts one particle (the one which is just spiking) to the level $0.$ As a consequence, the above transition does not create density -- and it even destroys smoothness due to the reset to $0$ of the spiking neuron. 
Finally, the only way that ``smoothness'' is generated by the process is the smoothness which is present in the ``noise of the jump times'' (which are basically of exponential density). For this reason, we have to stay away from the point $x=m,$ where the drift of the flow vanishes. Moreover, the reset-to-$0$ of the spiking particles implies that we are not able to say anything about the behavior of the invariant density of a single particle in $0 $ (actually, near to $0$) neither. Finally, we also have to stay strictly below the upper bound of the state space $K.$ That is why we introduce the following open set $ S_{d , \beta }$ given by
\begin{equation}\label{eq:estimset}
 S_{d , \beta } := \{ w \in [0, K] : \frac{\lfloor 
\beta \rfloor}{N} < w < K-   \frac{\lfloor 
\beta \rfloor}{N} , |w-m| > d \} ,
\end{equation}
where $\beta$ is the smoothness of the fixed class $ H( \beta , F, L, f_{min})  $ that we consider and where $ d $ is fixed such that $d > \frac{\lfloor 
\beta \rfloor +2}{N}.$ Notice that $S_{d , \beta } $ also depends on $K, m $ and $N$ which are supposed to be known. We are able to obtain a control of the invariant measure only on this set $S_{d , \beta }.$ The dependence in $\beta$ is due to the fact that the regularity of $f$ is transmitted to the invariant measure by the means of successive integration by parts (see \cite{evanew} for more details).

We quote the following theorem from \cite{evanew}.

\begin{theo}\label{theo:invmeasure} (Theorem 5 of \cite{evanew})

Suppose that $f \in H( \beta , F, L, f_{min} ).$ Let  
$$\pi_1   := {\mathcal L}_\pi  ( X_t^1) $$
be the invariant measure of a single neuron, \textit{i.e.}\ $ \int g d  \pi_1 = E_\pi ( g( X^1_t)   ) .$ 
Then $\pi_1 $ possesses a bounded continuous Lebesgue density $\pi^1 $ on $S_{d , \beta } $ for any $d$ such that $d > (\lfloor 
\beta \rfloor +2)/N ,$ which is bounded on $S_{d , \beta } ,$ uniformly in $f \in H( \beta , F, L, f_{min} ).$ 
Moreover, $\pi^1 \in C^k ( S_{d , \beta })  $ and 
\begin{equation}\label{eq:imctrl}
\sup_{\ell \le \lfloor 
\beta \rfloor  , w \in S_{d , \beta } } |  \pi_1^{(\ell )} ( w)  | + \sup_{w \neq w' , w, w' \in S_{d , \beta }} \frac{\pi_1^{(\lfloor 
\beta \rfloor )} (w) - \pi_1^{(\lfloor 
\beta \rfloor )} (w') }{|w-w'|^\alpha } \le C_F,
\end{equation}
where the constant $C_F $ depends on $d $ and on the smoothness class $ H( \beta , F, L, f_{min} ),$ but on nothing else. 
\end{theo}

\subsection{Statistical results}
We can now state the main theorem of our paper which describes the quality of our estimator in the minimax theory. We assume that $ m$ and $\lambda $ are known and that $f$ is the only parameter of interest of our model. We shall always write $ P_x^f $ and  $E_x^f$ in order to emphasize the dependence on the unknown $f.$ Fix some $r > 0 $ and some suitable point $a \in S_{d , \beta } .$ For any possible rate of convergence $(r_t )_{t \geq 0 } $ increasing to $\infty $ and for any process of ${\mathcal F}_t-$measurable estimators $\hat f_t $ we shall consider point-wise square risks of the type 
$$ \sup_{ f \in H( \beta , F, L, f_{min} )} r_t^2 E_x^f \left[ | \hat f_t ( a)  - f(a) |^2 | A_{t,r} \right] ,$$
where 
$$A_{t,r}:= \left\{\frac{1}{Nt} \int_0^t \int_{\R} Q_h(y-a) \eta(ds,dy) \geq r \right\}$$ 
is roughly the event ensuring that sufficiently many observations have been made near $a,$ during the time interval $[0, t ].$ We are able to choose $r$ small enough such that 
\begin{equation}\label{eq:noloss}
\underset{t \to \infty }{\lim\inf} \inf_{f \in H( \beta , F, L, f_{min} ) } P^f_x (A_{t,r} ) = 1,
\end{equation}
see Proposition \ref{prop:imlb} below.

Recall that the kernel $Q$ is chosen to be of compact support. Let us write $R$ for the diameter of the support of $Q,$ therefore $Q(x) =  0 $ if $ |x| \geq R.$  For any fixed $a \in S_{d , \beta } , $ write $h_0 := h_0 ( a, R, \beta , d )  := \sup \{ h > 0 : B_{ h R} (a) \subset S_{d/2, \beta } \} .$ Here, $ B_{h R } (a) = \{ y \in \R_+ : |y - a | <  h R \} .$  

\begin{theo}\label{theo:main}
Let $f \in  H( \beta , F, L, f_{min} )  $ and choose $Q \in C_{c} ( \R)$ such that $\int_{\R}  Q(y) y^j dy = 0 $ for all $ 1 \le j \le \lfloor \beta \rfloor ,$ and $ \int_{\R} |y|^\beta Q(y) dy < \infty .$ Then there exists $r^*>0$ such that the following holds for any $ a \in S_{d , \beta } , r \le r^* $ and for any $ h_t \le h_0.$ \\

(i) For the kernel estimate \eqref{eq:kernelest} with bandwidth $h_t = t^{ - \frac{1}{2 \beta + 1 } } ,$ for all $x \in [0,K],$
$$ \underset{t \to \infty }{\lim\sup} \sup_{ f \in H( \beta , F, L, f_{min} ) }t^{\frac{2 \beta }{2 \beta +1} } E_x^f \left[  | \hat f_{t, h_t} (a) - f(a) |^2  | A_{t,r}  \right]  < \infty .$$

(ii) Moreover, for $ h_t = o ( t^{ - 1 /(1 + 2 \beta ) } ) ,$ for every $f \in H( \beta , F, L, f_{min} ) $ and $ a \in S_{d , \beta } $
$$  \sqrt{th_t} \left( \hat f_{t,h_t} ( a) - f(a) \right) \to {\mathcal N} ( 0, \Sigma ( a) ) $$
weakly under $ P_x^f ,$ 
where $\Sigma ( a) =  \frac{f(a)}{N \pi_1 (a) } \int Q^2 (y) dy .$ 
\end{theo}

The next theorem shows that the rate of convergence achieved by the kernel estimate $ \hat f_{t, t^{- 1/(2 \beta + 1 ) } } $ is indeed optimal. 

\begin{theo}\label{theo:lowerbound}
Let $ a \in S_{d , \beta } $ and $x \in [0, K ] $ be any starting point. Then we have
\begin{equation}
\underset{t \to \infty }{\lim\inf} \inf_{ \hat f_t} \sup_{ f \in H( \beta , F, L, f_{min} ) }  t^{ \frac{2 \beta }{1 + 2 \beta } } E_x^f  [ | \hat f_t ( a) - f(a) |^2 ] > 0 ,
\end{equation}
where the infimum is taken over the class of all possible estimators $\hat f_t (a) $ of $f(a) .$ 
\end{theo}
The proofs of Theorems \ref{theo:main} and \ref{theo:lowerbound} are  given in Sections \ref{sec:proofmain} and \ref{sec:optimal}.

\subsection{Simulation results}
In this subsection, we present some results on simulations, for different jump rates $f.$ The other parameters are fixed: $N=100, \; \lambda=1, \; K=2$ and $m=1.$ The dynamics of the system are the same when $\lambda$ and $f$ have the same ratio. In other words, variations of $\lambda$ and $f$ keeping the same ratio between the two parameters lead to the same law for the process rescaled in time. This is why we fix $\lambda=1$ and propose different choices for $f.$
The kernel $Q$ used here is a truncated Gaussian kernel with standard deviation 1. 
%
%Since the assumption $\int_{\R}  Q(y) y^j dy = 0 $ is satisfied only for $j=1$ with this kernel, we chose $h_t=t^{-\frac{1}{3}}.$

We present for each choice of a jump rate function $f$ the associated estimated function $\hat f$ and the observed distribution of $X$ or more precisely of $\bar X=\frac{1}{N} \sum_{i=1}^N X^i.$ Figures 2, 3 and 4 correspond respectively  to the following definitions of $f: \; f(x)=x, \; f(x)= \log(x+1)$ and $f(x)= \exp(x)-1.$

For Figures \ref{Fig. 2}, \ref{Fig. 3} and \ref{Fig. 4}, we fixed the length of the time interval for observations respectively to $t=200,$ $300$ and $150.$ This allows us to obtain a similar number of jump for each simulation, respectively equal to $17324,$ $18579$ and $21214.$ These simulations are realized with the software R.

The optimal bandwidth $h_t=t^{-\frac{1}{2 \beta +1}}$ depends on the regularity of $f$ given by the parameter $\beta.$ Therefore, we propose a data-driven bandwidth chosen according to a Cross-validation procedure. For that sake, we define the sequence $\left( Z_k \right)_{k \in \N^*}$ by $Z_k^i=X_{T_k^i-}^i$ for all $1 \leq i \leq N.$ For each $a \in [0,K]$ and each sample $Z=(Z_1,...,Z_n),$ for $1\leq \ell\leq n$ we define the random variable $\hat \pi_1^{\ell, n,h} (a)$ by
$$
\hat \pi_1^{\ell, n,h} (a) = \frac{1}{(n-\ell)N} \sum_{k=\ell +1}^n \sum_{i=1}^N Q_h(Z_k^i-a).
$$
$\hat \pi_1^{\ell, n,h} (a)$ can be seen as an estimator of the invariant measure $\pi^Z_1$ of the discrete Markov chain.

We propose an adaptive estimation procedure at least for this simulation part. We use a Smoothed Cross-validation (SCV) to choose the bandwidth (see for example the paper of Hall, Marron and Park \cite{HaMa}), following ideas which were first published by Bowmann \cite{Bo} and Rudemo \cite{Ru}. As the bandwidth is mainly important for the estimation of the invariant probability $\pi_1^{Z} $, we use a Cross validation procedure for this estimation.  More precisely, we use a first part of the trajectory to estimate $\widehat \pi_1^{\ell, n,h}$ and then another part of the trajectory to minimize the Cross validation $SCV(h)$ in $h.$ In order to be closer to the stationary regime, we chose the two parts of the trajectory far from the starting time. Moreover we chose two parts of the trajectory sufficiently distant from each other. This is why we consider $m_1, m_2$ and $\ell$ such that $1<<m_1 \leq m_2 << \ell \leq n.$

%We will consider two different methods. In the first, we consider the minus the log-likelihood
%$$
%CV_1(h)=-\frac{1}{\tilde n N} \sum_{k=1}^{\tilde n} \sum_{i=1}^N \log \left( \hat \pi_1^{n,h} (\tilde Z_k^i) \right).
%$$ 
%{\color{red} Eva : je ne comprends pas : qui est $ \tilde Z_k^i ? $ et $\hat \pi_1^{n,h} $ n'est pas defini pour l'instant?}
%
%Minimizing the above quantity with respect to $h$ yields a first minimizer $h_1$. Notice that in the simulation we use a  simplified version of this estimator, since we do not care about the point where the log-likelihood is not defined as we use the convention of the programme R, for which $\log (0)=-\infty$.
%
%For the second method, 

We use the method of the least squares Cross validation and minimize 
$$
SCV(h)=\int \left( \widehat \pi_1^{\ell , n,h} (x) \right)^2 dx-\frac{2}{N(m_2-m_1)} \sum_{k=m_1+1}^{m_2} \sum_{i=1}^N \widehat \pi_1^{\ell, n,h} (Z_k^i)
$$
(where we have approximated the integral term by a Riemann approximation), giving rise to a minimizer $\hat h .$ We then calculate the estimator $\hat f$ the long of the trajectory. In the next figure, we use this method to find the reconstructed $ f $ with an adaptive choice of $h.$

\begin{figure}[!h]
\includegraphics[scale=0.6]{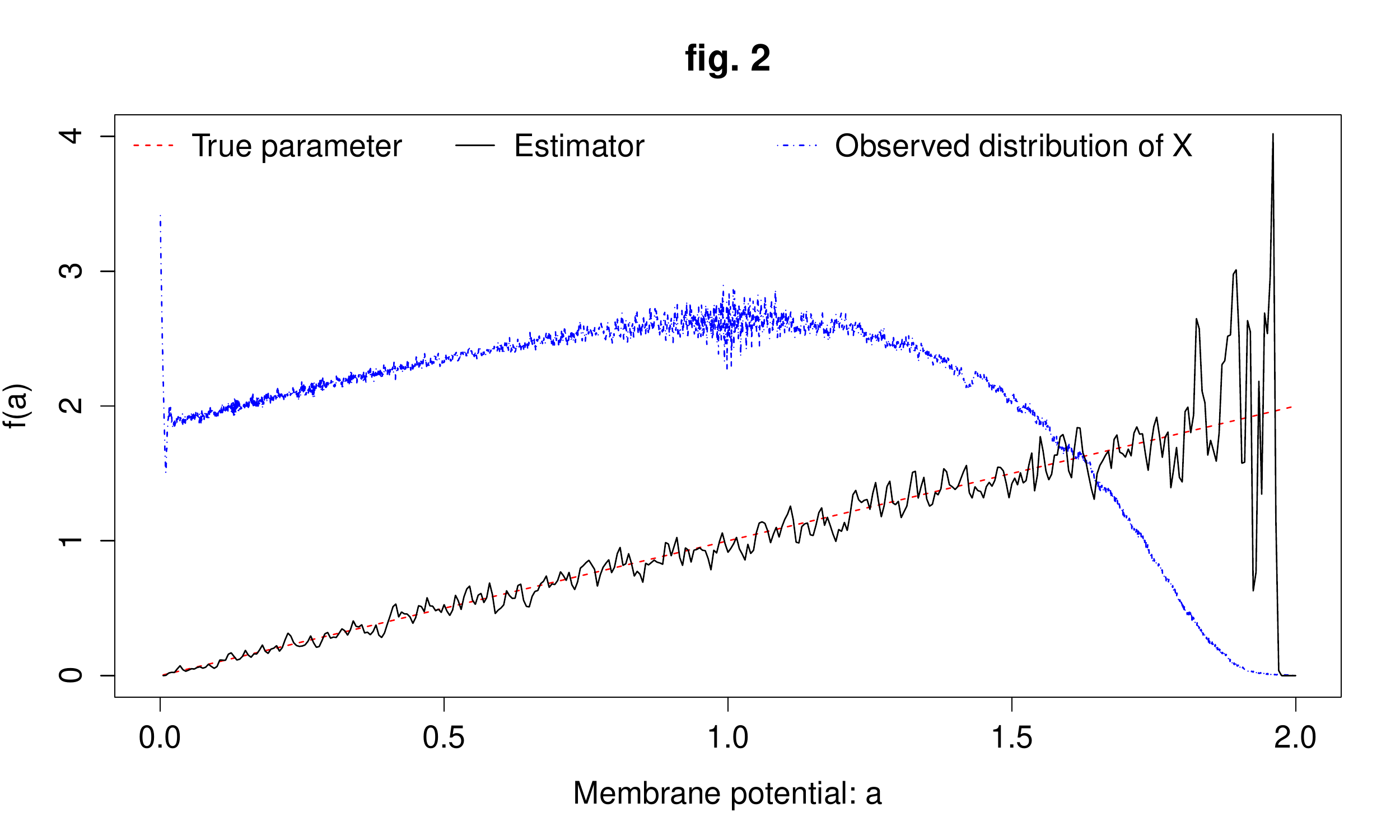} 
\caption{Estimation of the intensity function $f(x)=x$}
\label{Fig. 2}
\end{figure}

%On Figure \ref{Fig. 2}, it is interesting to see, that the 2 ways of finding the bandwith give differents results. Minimizing the log-likelihood gives a smoother estimator.

\begin{figure}[!t]
\includegraphics[scale=0.56]{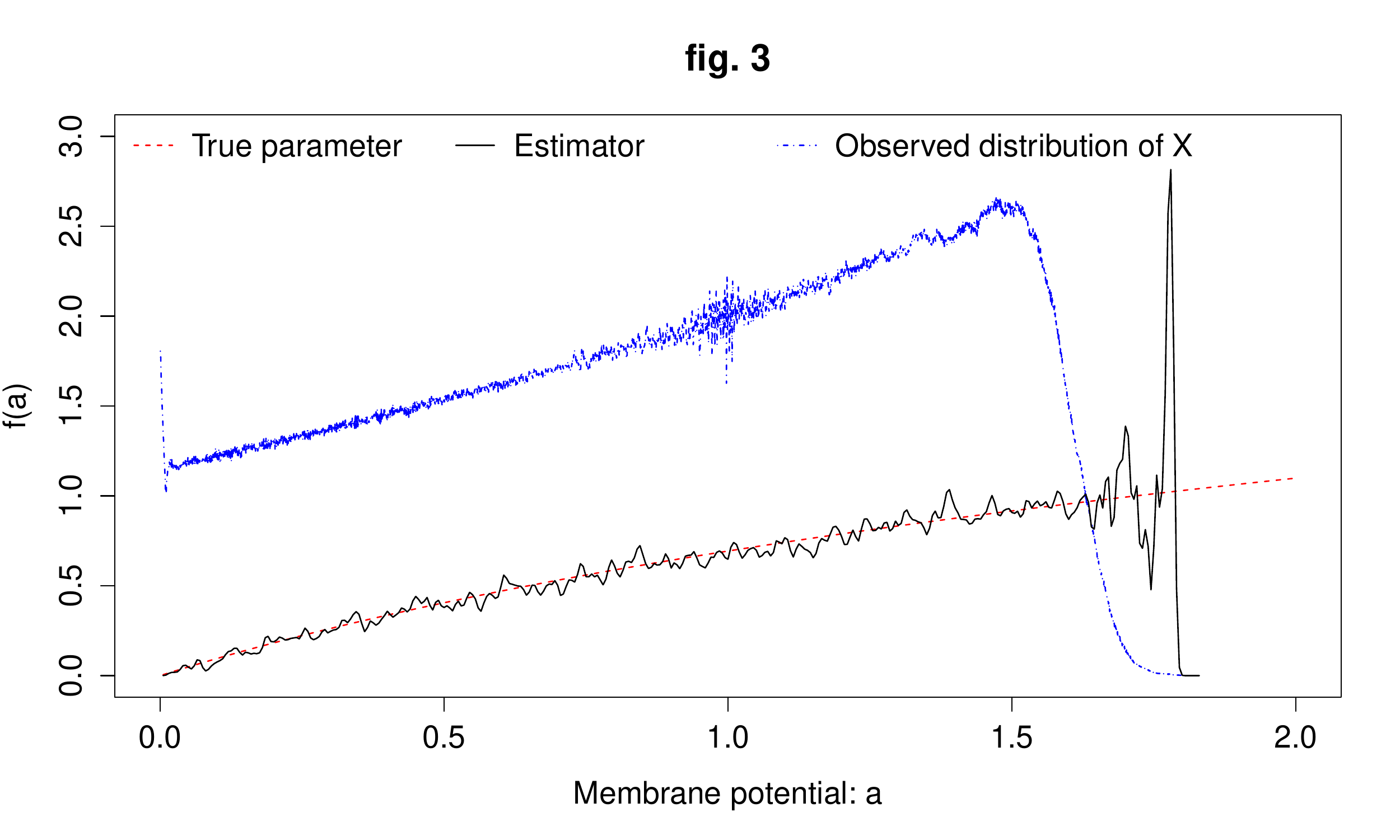}  
\caption{Estimation of the intensity function $f(x)=\log (x+1)$}
\label{Fig. 3}
\end{figure}

\begin{figure}[!h]
\includegraphics[scale=0.56]{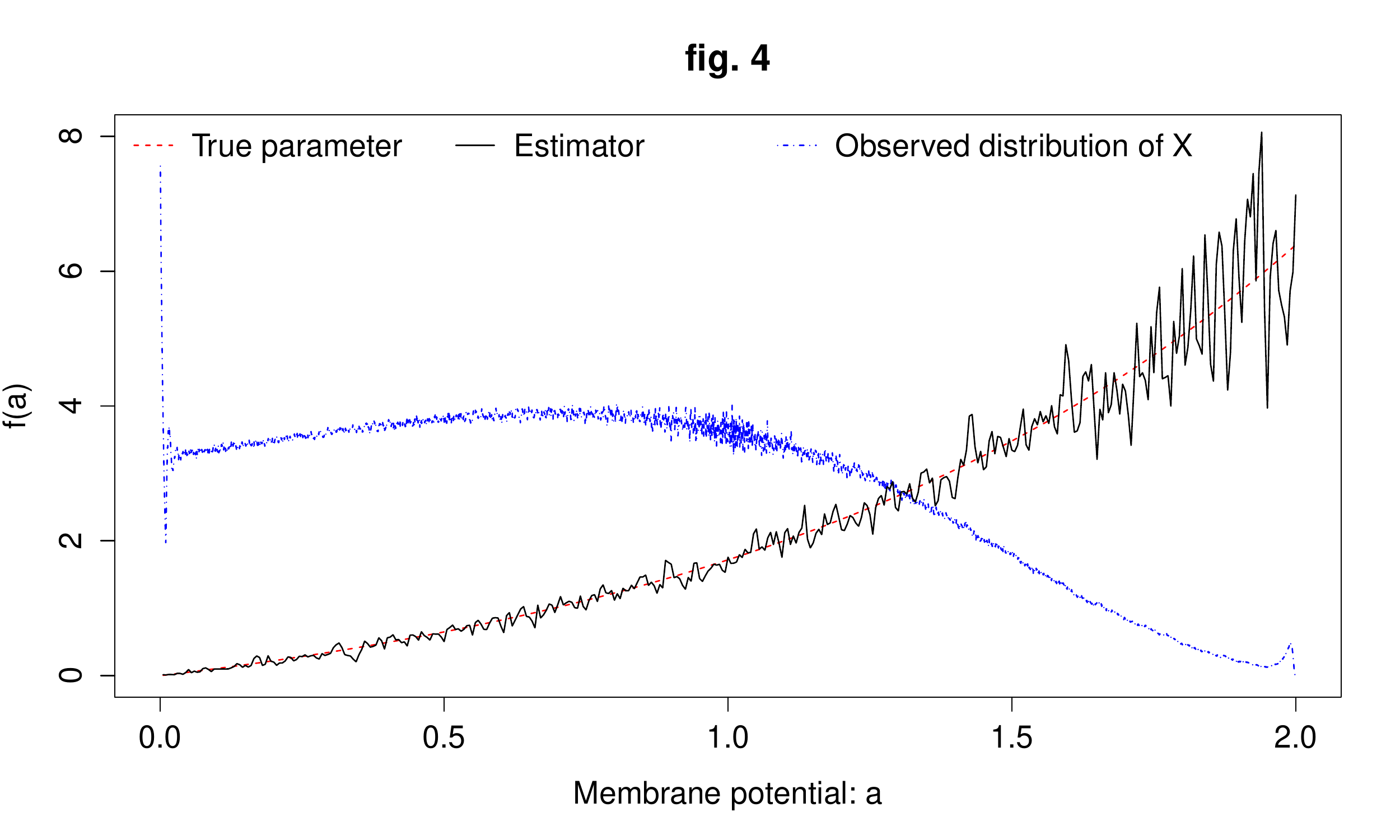} 
\caption{Estimation of the intensity function $f(x)=\exp (x) - 1$}
\label{Fig. 4}
\end{figure}

As expected, we can see that the less observations we have, the worse is our estimator. Note that close to $0$ the observed density of $X$ explodes. This was also expectable due to the reset to $0$ of the jumping neurons. Moreover, the simulations show a lack of regularity of the observed density close to $m,$ which is consistent with our results, but this does not seem to affect the quality of the estimator.

\section{Harris recurrence of $X$ and speed of convergence to equilibrium -- Proof of Theorem \ref{theo:harrisok}}\label{sec:Harris}
In this section, we give the proof of Theorem \ref{theo:harrisok} and show that the process $ (X_t)_{t \geq 0} $ is positive recurrent in the sense of Harris. We follow a classical approach and prove the existence of regeneration times. This is done in the next subsection and follows ideas given in Duarte and Ost \cite{bresiliens}.

\subsection{Regeneration}
The main idea of proving a regeneration property of the process is to find some uniform ``noise'' for the whole process on some ``good subsets'' of the state space. Since the transition kernel associated to the jumps of our process is not creating any density (and actually destroys it for the spiking neurons which are reset to $0$), the only source of noise is given by the random times of spiking. These random times are then transported through the 
deterministic flow $ \gamma_{s,t} (  v ) = ( \gamma_{s, t } ( v^1 ) , \ldots , \gamma_{s, t } ( v^N) ),$ which is given for any starting configuration $v \in [0, K]^N$ by
\begin{equation}\label{eq:flow}
\gamma_{s,t} (  v^i) = e^{ -\lambda (t-s) }  v^i + ( 1 - e^{ - \lambda (t-s) } ) m  , \; 0 \le s \le t , \; \gamma_t ( v^i ) := \gamma_{0, t } ( v^i ) .
\end{equation}

The key idea of what follows -- which is entirely borrowed from \cite{bresiliens} -- is the following. 

Write $I_n, n \geq 1, $ for the sequence giving the index of the spiking neuron at time $T_n,$ \textit{i.e.} $I_n = i $ if and only if $ T_n = T_k^i $ for some $ k \geq 1 .$ It is clear that in order to produce an absolute continuous law with respect to the Lebesgue measure on $ [0, K ]^N, $ we need at least $N$ jumps of the process. On any event of the type $ \{ T_1 = t_1, I_1 = i_1 ,  \ldots , T_N = t_N , t_N < t <  T_{N+1} , I_N = i_N \} ,$ it is possible to write the position of the process at time $t $ as a concatenation of the deterministic flows given by 
\begin{equation}\label{eq:concatenation}
\Gamma_{ (t_1, \ldots, t_N, i_1, \ldots , i_N )} ( t, v) = \gamma_{t_{N}, t}  ( \Delta_{i_N} ( \gamma_{t_{N-1} , t_N} ( \Delta_{i_{N-1} } ( \ldots   \Delta_{i_1} ( \gamma_{0, t_1} ( v) ) ) ) ) ) .
\end{equation}
Proving absolute continuity amounts to prove that the determinant of the Jacobian of the map $ (t_1, \ldots , t_N ) \to \Gamma_{ (t_1, \ldots, t_N, i_1, \ldots , i_N )} ( t, v) $ does not vanish. For general sequences of $ (i_1, \ldots , i_N) ,$ this will not be true (think \textit{e.g.}\ of the sequence $ (i_1 = \ldots = i_N = 1)$). 

The main idea is however to consider the sequence $ i_1 = 1, i_2 = 2, \ldots , i_N = N $ and to use the {\it regeneration property of spiking}, \textit{i.e.}\ the fact that the neuron $k$ spiking at time $t_k$ is reset to zero at time $t_k.$ In this case, for all later times, its position does not depend on $ t_1, \ldots , t_{k-1} $ any more. In other words, the Jacobian of  $\Gamma_{ (t_1, \ldots, t_N, 1, \ldots , N )} ( t, v)  $ is a diagonal matrix, and all we have to do is to control that all diagonal elements do not vanish. The second idea is to linearize the flow, \textit{i.e.}\ to consider the flow during very short time durations, and to use that, just after spiking, each diagonal element is basically of the form 
$$ \frac{\partial \gamma_{s, t } (0) } { \partial s } \sim - \lambda m  , \; \mbox{ as $t -s \to 0.$}$$ 
The important fact here is that the absolute value of the drift term of the deterministic flow of one neuron is strictly positive when starting from the initial value $0.$

In the following, this idea is made rigorous. Our proof follows the approach given  in Section 4 of \cite{bresiliens}. 
We fix $ \ge > 0 $ and put
$$ A_\ge = \{ i \ge - \ge  /  4  < T_i < i \ge , i = 1 , \ldots, N \} $$
and
$$ S = \{ I_1 =1, I_2 = 2 , \ldots , I_N = N  \}$$
which is the event that all $N$ neurons have spiked in the fixed order given by their numbers, \textit{i.e.}\ neuron $1$ spikes first, then neuron $2,$ then $3,$ and so on. We introduce 
$$ u^* = \left( \frac{N-1}{N}, \frac{N-2}{N}, \ldots, \frac1N , 0 \right) $$
which would be the position of neurons after $N$ spikes and on the event $S,$ if  $\lambda = 0 $ (here, we suppose w.l.o.g. that $ K > 1 + \frac1N $). 

Now we fix any initial configuration $v \in [0, K]^N $ and introduce the sequence of configurations
$ v (k) , 0 \le k \le N ,$ given by $v(0) = v,$ $v_k(k) = 0 $ and 
\begin{equation}\label{eq:vk}
   v_i(k) = \left\{
\begin{array}{ll}
\frac{k-i}{N} ,& i < k \\
\underbrace{a_K \circ \ldots \circ a_K}_{k \; \mbox{{\small times}}} ( v_i) , & i > k 
\end{array}
\right\} .
\end{equation}
Notice that $ \underbrace{a_K \circ \ldots \circ a_K}_{k \; \mbox{{\small times}}} ( v_i)  = v_i +  \frac{k}{N} $ if $ v_i < K - \frac{2+k}{N} .$ Notice also that  $v(N) = u^* .$ 

We cite the following lemma from \cite{bresiliens}.
\begin{lem}[Lemma 4.1 of \cite{bresiliens}]\label{lem:positionbresiliens}
If $ X_0 = u  \in B_\delta ( v) , $ then on the event $ A_\ge \cap S , $ we have for all $ 1 \le k \le N , $ \\
(i) $X_i ( T_k) = v_i ( k) + \sum_{ r= i+1}^k \lambda (T_r - T_{r-1} ) d_i (r-1)  + R_{  \delta \ge} ( T_1^k , u )  + R_{ \ge^2 }( T_1^k , u ) , $ if $i < k ,$ \\

(ii) $X_i ( T_k) = v_i ( k) + \sum_{ r= i}^k \lambda (T_r - T_{r-1} ) d_i  (r-1) +R_{\delta } ( u)  +  R_{\delta \ge} ( T_1^k , u )  + R_{\ge^2 }(T_1^k , u ) , $ if $i > k ,$ \\

(iii) $ \bar X^N ( T_k ) = \bar v (k) + R_{\delta } (u)  + R_{ \ge }( T_1^k , u)  , $ if $ k < N ,$ and $  \bar X^N ( T_N ) = \bar u^* + R_{\ge }( T_1^N , u )  .$

Here, $ d_i ( r) = m - v_i ( r)  $ and $ T_1^k = (T_1, \ldots, T_k).$ Moreover, the remainder functions are of order 
$$  R_{  \delta \ge} ( T_1^k , u ) = O ( \delta \ge) ,  R_{ \ge^2 }( T_1^k , u )  = O ( \ge^2 ) , R_{\delta } ( u) = O ( \delta ) ,\ldots ,$$ 
and all partial derivatives are of order either $\delta $ or $\ge ,$ uniformly in $ v .$ 
\end{lem}

\begin{rem}
Our model is slightly different from the model in \cite{bresiliens}: instead of an attraction to the empirical mean of the system, we have an attraction to a fixed equilibrium value $m.$ This leads to our definition of $ d_i ( r) $ which is slightly different from the one used in \cite{bresiliens}. %The other difference is the boundedness of the state space of our model, but here we are considering positions $v_i(k)$ far from the bound $K,$ where the dynamics are the same as in \cite{bresiliens}.
%With these conditions, it is easy to see that the result still holds in our case.
\end{rem}

\begin{cor}[\cite{bresiliens}, Corollary 3]
Put $ t^* = N \ge .$ Then we have on $ A_\ge \cap S, $ 
\begin{equation}\label{eq:14}
X_i ( t^*) = u_i^* + \lambda (t^* - T_N) d_i^* + \sum_{ r= i+1}^N \lambda (T_r - T_{r-1} ) d_i (r-1)  + R_{  \delta \ge} ( T_1^N , u )  + R_{ \ge^2 }( T_1^N , u ) ,
\end{equation} 
where $ d_i^* = m - u^*_i .$ 
\end{cor}

We put as in \cite{bresiliens} $\gamma^0 (t_1^N ) = ( \gamma^0_1 ( t_1^N) , \ldots , \gamma^0_N ( t_1^N) ) ,$ where 
$$ \gamma^0_i ( t_1^N ) := u_i^* + \lambda (t^* - t_N) d_i^* + \sum_{ r= i+1}^N \lambda (t_r - t_{r-1} ) d_i (r-1) , 1 \le i \le N.$$
Hence $\gamma^0_i (t_1^N) $ models how the $N$ successive jump times $ t_1 < t_2 < \ldots < t_N$ are mapped, through the deterministic flow, into a final position at time $t^* -$ on the event $\{ T_1 = t_1, \ldots , T_N = t_N \} \cap  A_\ge  \cap S.$ In order to control how the law of the $ N$ successive jump times $ t_1 , \ldots , t_N $ is transported through this flow, we calculate the partial derivatives of $ \gamma^0 $ with respect to $ t_i, 1 \le i \le N.$ One sees immediately that 
$$ \frac{\partial \gamma^0_i}{\partial t_k} = 0 , k < i , \quad \frac{\partial \gamma^0_i}{\partial t_i} = - \lambda m , 1 \le i \le N ,$$
whence

\begin{cor}[Corollary 4 of \cite{bresiliens}]
For each $ u \in B_\delta ( v) , $ the determinant of the Jacobian of the map $ \{ i \ge - \ge  /  4  < t_i < i \ge , i = 1 , \ldots, N \}   \ni t_1^N \mapsto \gamma^0 (t_1^N) +  R_{  \delta \ge} ( t_1^N , u )  + R_{ \ge^2 }( t_1^N , u )  $ is given by 
$$ \lambda^N m^N  + R_\ge (t_1^N , u) + R_\delta ( t_1^N , u) $$
which is different from zero for $ \ge $ and $\delta $ small enough, for all $ u \in B_\delta ( v) .$ 
\end{cor}

As in Proposition 4.1 of \cite{bresiliens}, we now have two important conclusions from the above discussion. 

\begin{prop}
There exists $ \delta^* > 0$ and $\ge > 0,$ such that  for $ t^* = N \ge ,$ 
\begin{equation}\label{eq:regen}
 P_{t^* } ( x, \cdot ) \geq \eta_1 1_{ B_{\delta^*} ( u^*) } ( x) \nu ,
\end{equation}
where $ \nu $ is a probability measure and $\eta_1 \in ]0, 1 [.$ 
\end{prop}

 The lower bound \eqref{eq:regen} is a local Doeblin condition, and its proof is given in Proposition 4.1 of \cite{bresiliens}. We call $B_{\delta^*} ( u^*)$ a regeneration set: if the process visits this regeneration set, then after a time $t^*$ there is a probability $\eta_1$ that the law of the process is independent from its initial position $x \in B_{\delta^*} ( u^*).$

To be able to make use of the local Doeblin condition, we have to be sure that the process actually does visit the regeneration set $ B_{\delta^*} ( u^*).$ This is granted by the following result.

\begin{prop}\label{prop:control}
There exist $ \ge > 0 $ and $ \eta_2 > 0$ such that 
$$\inf_{f \in H( \beta , F, L, f_{min} )} \; \inf_{ v \in [0,K]^N }  P_{t^* } ( v , B_{\delta^*} ( u^* ) ) \geq \eta_2 ,$$
 for $ t^* = N \ge .$ 

\end{prop} 

\begin{proof}
By (\ref{eq:14}), there exists $ \ge $ such that for all $  v  \in [0,K]^N, $ we have that $X(t^* ) \in B_{\delta^*  }( u^* )$ on $A_\ge \cap S ,$ when $ X( 0 ) =  v.$ Hence  
$$  P_{t^* } ( v , B_{\delta^*} ( u^* ) )  \geq  P_v({A_\ge \cap S} )  .$$

Recalling (\ref{eq:flow}), we then obtain 
\begin{multline*}
 P_v({A_\ge \cap S} )  = \int_{ \ge - \ge / 4    }^{\ge} f( \gamma_{0, t_1} ( v^1)  ) e^{ -  \int_0^{t_1}  \bar f( \gamma_{0,s} ( v)  )ds } dt_1 \\
 \int_{ 2 \ge  - \ge / 4   }^{2 \ge}f( \gamma_{ t_1, t_2} ( v (1)^2 )  ) e^{ - \int_{t_1}^{t_2}  \bar f( \gamma_{t_1,s}( v (1) )  )ds } dt_2 \ldots \\
\int_{ N \ge - \ge / 4   }^{N \ge}f( \gamma_{ t_{N-1}, t_N} ( v (N-1)^N )  ) e^{ -  \int_{t_{N-1}}^{t_N}  \bar f( \gamma_{t_{N-1},s} ( v (N-1)) )  )ds } dt_N ,
\end{multline*}
where $ \bar f ( x) = \sum_{i=1}^N f( x^i) , $ where the sequence $ v(1) , \ldots , v( N-1) $ is given as in (\ref{eq:vk}). Since by assumption $ v \in [0,K]^N ,$ it is immediate to see that $ \gamma_{s, t } ( v (k)^i ) \le  C, $ for a constant $C,$ for all $ 0 \le s \le t \le t^* , $ for all $ k \le N$ and for all $i \le N.$  Moreover, 
$$ \gamma_{0, t_1}^1 ( v^1)  \geq (1- e^{ - \lambda \frac34 \ge  } ) m > 0,  \mbox{ on } t_1 \geq \frac34  \ge , $$ 
and since $f$ is non decreasing, satisfying $ f (x) \geq f_{min}( x) $ and $ \| f \|_\infty \le F,$  this implies that 
$$ f( \gamma_{0, t_1} ( v^1)  ) e^{ -  \int_0^{t_1}  \bar f( \gamma_{0,s}( v)  )ds } \geq f( (1- e^{ - \lambda \frac34 \ge  } ) m) e^{ - N t_1 f( C ) } \geq f_{min} \left( (1- e^{ - \lambda \frac34 \ge  } ) m \right) e^{ - N t_1 F }, $$
on $ t_1 \geq  \frac34  \ge .$ Similar arguments show that all consecutive terms are strictly lower bounded uniformly in $f \in H( \beta , F, L, f_{min} ) $ as well. 
As a consequence, 
$$  P_v({A_\ge \cap S} )  \geq \left( \frac{\ge}{4} f_{min} \left( (1- e^{ - \frac34 \lambda \ge} )  m \right) e^{ -  t^*N  F }  \right)^N > 0,$$
which concludes the proof. 
\end{proof}

\begin{rem}\label{rq:regenunif}
In the proof of Proposition 4.1 of \cite{bresiliens}, the authors have no need to obtain (\ref{eq:regen}) uniformly in $f \in H( \beta , F, L, f_{min} ).$ However, it is easy to see that we can rewrite their proof using the bounds for $ f \in H( \beta , F, L, f_{min} ) $ appearing in the proof of Proposition \ref{prop:control} above. As a consequence, we obtain
\begin{equation}\label{eq:regenunif}
 \inf_{f \in H( \beta , F, L, f_{min} )} P_{t^* } ( x, \cdot ) \geq \eta_1 1_{ B_{\delta^*} ( u^*) } ( x) \nu ,
\end{equation}
for some $\eta_1 > 0.$
\end{rem}

Once we dispose of the uniform local Doeblin condition \eqref{eq:regenunif} and of the control given in Proposition \ref{prop:control}, it is classical, using regeneration arguments, to show that the process is recurrent in the sense of Harris. 

\subsection{Harris recurrence and invariant measure}

%Let us recall the notion of Harris recurrence of a process $X.$
%\begin{defin}\label{def:HR} 
%$X$ is called \textup{Harris recurrent} if there exists some $\sigma$-finite measure $ \mu $ on $[0, K]^N$ such that for all $ B \in %{\mathcal B} ( [0, K ]^N) , $ 
%\begin{equation}\label{eq:definrec}
%\mu (B) > 0 \; \mbox{ implies }    P_x \left( \int_0^\infty 1_B (X_s) ds = \infty \right) =1 
%\end{equation}
%for all $ x \in [0, K]^N  .$ Moreover, the process $X$ is called \textup{positive Harris recurrent}, if \eqref{eq:definrec} holds for a %finite measure $ \mu $ on $[0, K]^N .$ 
%\end{defin} 

Using the regeneration procedure, we can prove that the process $X$ is positive Harris recurrent. We denote by $ \| \cdot \|_{TV} $ the total variation distance, \textit{i.e.}\ $\| \nu_1 - \nu_2 \|_{TV} = \sup_{ B \in {\mathcal B} ( [0, K ]^N ) } \left| \nu_1 ( B) - \nu_2 (B) \right| , $ for any two probability measures $\nu_1 , \nu_2 $ on $ ( [0, K ]^N , {\mathcal B} ( [0, K ]^N )) .$

%The proof of the above theorem is given in Appendix. The above statement implies that the process is uniformly (and thus exponentially) ergodic, having invariant probability measure $ \mu .$ 

We first show that the process is indeed Harris. For that sake, define the sequence of stopping times $(\tilde S_n )_{n \in \N}$ 
$$
\tilde S_1:= \inf \{ t  > 0 : X_t \in B_{\delta^*} ( u^*) \},
$$ 
and for all $n \geq 1,$ 
$$ 
\tilde S_{n+1}:= \inf \{ t > \tilde S_n + t^* : X_t \in B_{\delta^*} ( u^*) \}.
$$

Let $(U_n )_{n \in \N}$ be a sequence of \textit{i.i.d.}\ uniform random variables on $[0,1],$ which are independent of the process $X.$ 
Then, working conditionally on the realization of  $(U_n )_{n \in \N},$ we define the sequence $ ( S_n )_{ n \in \N }$ and the sequence $(R_n )_{n \in \N}$ of regeneration times as follows. 

$$
S_1:= \inf \{ \tilde S_n : U_n < \eta_1 \}, R_1 := S_1 + t^* ,
$$ 
and for all $n \geq 1,$ 
$$
S_{n+1}:= \inf \{ \tilde S_k > S_n : U_k < \eta_1 \}, R_{n+1} := S_{n+1} + t^* ,
$$
where $\eta_1 $ is given in \eqref{eq:regenunif}.
  
\begin{rem}\label{rem:iidregen}
(\ref{eq:regenunif}) allows us to construct the process $(X_t)_{t \geq 0} $ on a bigger probability space in such a way that for all $n, X_{R_n} \sim \nu$ and $\left(X_{R_n+t}\right)_{t \in \R^+}$ is independent from $ {\mathcal F}_{S_n-}. $  This construction is known as Nummelin splitting, we refer the interested reader to Chapter 6  of L\"ocherbach (2013) \cite{cours}. 
 \end{rem}

\begin{lem}\label{lem:FTregen}
For all $x \in [0, K]^N , E_x (R_1) < \infty$ and $E_x (R_2-R_1) < \infty.$
\end{lem}

The proof of this lemma is postponed to the next subsection where we prove a stronger result. Now the following result implies that our process is actually positive Harris recurrent.

\begin{prop}\label{prop:HR}
$X$ is Harris recurrent with invariant probability measure $ \pi  $ which is given by 
$$
\pi (B):=\frac{1}{ E_x ( R_2 - R_1) }  E_x \left( \int_{R_1}^{R_2} 1_B(X_s) ds \right) .
$$
\end{prop}

\begin{proof}
Fix $B \in {\mathcal B} (  [0, K]^N)$ and define the process $A_t$ by 
$$
A_t:= \int_0^t 1_B(X_s)ds.
$$
Assume that $\pi(B)>0,$ then, according to the definition of Harris recurrence, it is enough to show that for all $x, \lim_{t \to +\infty} \frac{A_t}{t}>0.$

We denote by $\tilde N_t$, $\tilde N^e_t$ and $\tilde N^o_t$ the counting processes respectively associated with the sequences of stopping times $(R_n)_{n \in \N^*}, (R_{2n})_{n \in \N^*}$ and $(R_{2n+1})_{n \in \N}:$

$$
\tilde N_t := \sum_{n=1}^\infty 1_{R_n \leq t}, \; \; \tilde N_t^e := \sum_{n=1}^\infty 1_{R_{2n} \leq t} \; \mbox{and} \; \tilde N_t^o := \sum_{n=0}^\infty 1_{R_{2n+1} \leq t}.
$$

For all $t$ we have  $\tilde N_t=\tilde N^e_t + \tilde N^o_t$ and
\begin{multline*}
\frac{A_t}{t}=\frac{1}{t} \left( \int_0^{R_1} 1_B(X_s)ds + \sum_{n=1}^{\tilde N_t} \int_{R_n}^{R_{n+1}} 1_B(X_s)ds - \int_{R_{\tilde N_t}}^t 1_B(X_s)ds \right)
\\
=\frac{1}{t} \left( \int_0^{R_1} 1_B(X_s)ds + \sum_{k=1}^{\tilde N^e_t} \int_{R_{2k}}^{R_{2k+1}} 1_B(X_s)ds + \sum_{k=1}^{\tilde N^o_t} \int_{R_{2k-1}}^{R_{2k}} 1_B(X_s)ds - \int_{R_{\tilde N_t}}^t 1_B(X_s)ds \right).
\end{multline*}

When $t$ goes to $\infty,$ we obtain, using Lemma  \ref{lem:FTregen} to deal with the first and the last terms,

\begin{multline*}
\lim_{t \to +\infty} \frac{A_t}{t}
\\
= \lim_{t \to +\infty} \frac{\tilde N_t}{t} \frac{1}{\tilde N_t} \left( \tilde N^e_t \left( \frac{1}{\tilde N^e_t} \sum_{k=1}^{\tilde N^e_t} \int_{R_{2k}}^{R_{2k+1}} 1_B(X_s)ds \right) + \tilde N^o_t \left( \frac{1}{\tilde N^o_t} \sum_{k=1}^{\tilde N^o_t} \int_{R_{2k-1}}^{R_{2k}} 1_B(X_s)ds \right) \right).
\end{multline*}

The decomposition between even and odd regeneration times is used here to be able to apply the strong law of large numbers, based on Remark \ref{rem:iidregen}. In this way we obtain that 
$$
\lim_{t \to +\infty} \frac{1}{\tilde N^e_t} \sum_{k=1}^{\tilde N^e_t} \int_{R_{2k}}^{R_{2k+1}} 1_B(X_s)ds= \lim_{t \to +\infty} \frac{1}{\tilde N^o_t} \sum_{k=1}^{\tilde N^o_t} \int_{R_{2k-1}}^{R_{2k}} 1_B(X_s)ds  = \pi(B) > 0  \; \mbox{a.s.}
$$
We can use the same decomposition to obtain that 
$$
\lim_{t \to +\infty} \frac{\tilde N_t}{t} = \frac{1}{E_x(R_2-R_1)} \; \mbox{a.s.}
$$
Putting all together we have 
$$
\lim_{t \to +\infty} \frac{A_t}{t}=\frac{\pi(B)}{E_x(R_2-R_1)}
$$
and we can conclude the proof using Lemma \ref{lem:FTregen} once again.
\end{proof}

\subsection*{Speed of convergence to equilibrium -- Proof of (\ref{eq:ergodic}) in Theorem \ref{theo:harrisok}}
We now show how to couple two processes $X$ and $Y$ following the same dynamics (\ref{eq:dyn}) using Proposition \ref{prop:control} and the lower bound (\ref{eq:regenunif})  of Remark \ref{rq:regenunif}. This coupling will give us a control of the distance in total variation between $P_x$ and $P_y,$ where $x$ and $y$ are the respective starting points of processes $X$ and $Y.$
  
The coupling procedure consists in using the same realization of uniform random variables $(U_n)_{n \in \N } $ for both processes, relying on (\ref{eq:regenunif}),  when both processes $X$ and $Y$ are in the regeneration set $B_{\delta^*} ( u^*)$ at the same time. More precisely, we let evolve $X$ and $Y$ independently up to the first time that they are both in the set $ B_{\delta^*} ( u^*).$ We introduce the sequence of stopping times 
$$ \bar S_1 = \inf \{ t > 0 : (X_t , Y_t) \in B_{\delta^*} ( u^*) \times B_{\delta^*} ( u^*) \} $$
and 
$$ \bar S_{n} = \inf \{ t > \bar S_{n-1} + t^* : (X_t , Y_t) \in B_{\delta^*} ( u^*) \times B_{\delta^*} ( u^*) \} ,\;  n \geq 1.$$ 

Applying Proposition \ref{prop:control} to two independent processes $X$ and $Y,$ we obtain

\begin{equation}\label{eq:alignment}
\inf_{f \in H( \beta , F, L, f_{min} )} \; \inf_{ v_1 , v_2 \in [0,K]^N  }  P^{\otimes 2}_{t^* } ( ( v_1 , v_2) , B_{\delta^*} ( u^* )^2 ) \geq \eta_2^2.
\end{equation}

As a consequence, $ \bar S_n < \infty $ almost surely for all $n,$ and $ P_{(v_1, v_2)}  ( \bar S_1 > n t^* ) \le (1 - \eta_2^2 )^n ,$ \textit{i.e.}\ $\bar S_1 $ and $ \bar S_{n+1} - \bar S_n$ possess exponential moments 
$$E_{(v_1, v_2) } [ e^{ \alpha \bar S_1 } ] < \infty $$
uniformly in the starting configuration $ (v_1, v_2) $ for all $ \alpha  < \frac{- \ln ( 1 - \eta_2^2 ) }{t^*} .$

We are now able to couple the processes $X$ and $Y.$ We work conditionally on the realization of a sequence of \textit{i.i.d.}\  uniform random variables $(U_n)_{n \in \N } $ and define the coupling time $\tau$ by
$$
\tau:= \inf \{ \bar S_n : U_n  \le \eta_1 \} + t^* .
$$
Using the regenerative construction described in the previous subsection based on (\ref{eq:regenunif}), it is evident that $X$ and $Y$ can be constructed jointly in such a way that $ X_\tau = Y_\tau \sim \nu $ and such that $ X_t = Y_t $ for all $t \geq \tau .$  Since $\tau $ is constructed by sampling within the sequence $ (\bar S_n)_{n \in \N} $ at an independent geometrical time, it is immediate to see that there exists $\kappa > 1 $ such that 
\begin{equation}\label{eq:ctmoment}
\sup_{ v_1, v_2 \in [0, K]^N } E_{(v^1,v^2)}(\kappa^{\tau}) < + \infty.
\end{equation} 

\begin{rem}
Notice that the regeneration time $R_1$ can be compared to $\tau $ and that $R_1 \le \tau .$  As a consequence, \eqref{eq:ctmoment} implies a proof of Lemma \ref{lem:FTregen}. 
\end{rem}

Since the two processes $X$ and $Y$ follow the same trajectory after time $\tau, $ we obtain the following classical upper bound on the total variation distance.
\begin{equation}\label{eq:couplingspeed}
\| P_t(x,\cdot ) -P_t(y,\cdot ) \| _{TV} \leq P_{(x,y)}(\tau >t) \leq \kappa^{-t} E_{(x,y)}( \kappa^{\tau}).
\end{equation}

Now putting $C:=\sup_{x , y \in [0,K]^N} E_{(x, y )}( \kappa^{\tau}),$ the integration of (\ref{eq:couplingspeed}) with respect to the invariant measure $\pi (dy)$ implies that 
\begin{equation*}
\sup_{x \in [0,K]^N} \left\Vert P_t(x,\cdot ) - \pi \right\Vert_{TV} \leq C \kappa^{-t}.
\end{equation*} 
This finishes the proof of Theorem \ref{theo:harrisok}. \hfill $\qed$

\subsection{Estimates on the invariant density of a single particle}

We start with some simple preliminary estimates. Recall that
$$ \mu ( ds, dx) = \sum_n \delta_{T_n } (ds) \delta_{X_{T_n-}} (dx) $$
denotes the jump measure of the system, with compensator 
$$ \hat \mu  ( ds, dx ) = \bar f (X_s ) ds \delta_{X_s} (dx), \mbox{ with } \bar f (x) = \sum_{i=1}^N f (x^i ).$$
Let $Z_k = X_{T_k- } , k \geq 1 , $ be the jump chain. Then the following holds.

\begin{prop}\label{prop:3}
$(Z_k)_k$ is Harris recurrent with invariant measure given by 
$$ \pi^Z ( g) = \frac{1}{\pi ( \bar f ) } \pi ( \bar f g ) , $$
for any $g : \R_+^N \to \R$ measurable and bounded.
\end{prop}

\begin{proof}
Let $ g$ be a bounded test function. We have to prove that
$$ \frac{1}{n} \sum_{k=1}^n g ( Z_k) \to \pi^Z ( g) $$
as $n \to \infty, $ $P_x-$almost surely, for any fixed starting point $x \in [0, K]^N.$ But
$$ \frac{1}{n} \sum_{k=1}^n g ( Z_k) = \frac{1}{n} \sum_{k=1}^n g ( X_{T_k-}),$$
and, putting $ N_t = \sup \{ n : T_n \le t \},$  
$$
 \lim_{ n \to \infty }  \frac{1}{n} \sum_{k=1}^n g ( X_{T_k-}) = \lim_{t \to \infty}\frac{t}{N_t} \frac{1}{t} \sum_{k=1}^{N_t} g ( X_{T_k-}) 
= \lim_{t \to \infty}\frac{t}{N_t} \frac{1}{t} \int_0^t\int_{\R_+^N} g (x) \mu ( ds, dx) .
$$
By the law of large numbers, $ N_t/t \to \int \bar f(x) \mu (dx ) = \mu ( \bar f ) ,$ and this convergence holds almost surely. Moreover, 
\begin{equation}\label{eq:martingaleplusbiaisconv}
 \frac{1}{t} \int_0^t\int_{\R_+^N} g (x) \mu ( ds, dx)  = \frac{1}{t} M_t + \frac{1}{t} \int_0^t\int_{\R_+^N} g (x) \hat \mu ( ds, dx) ,
\end{equation}
where $ M_t = \int_0^t \int g ( x) [\mu (ds, dx) - \hat \mu  (ds, dx) ].$ Then $M_t$ is in ${\mathcal M}^{2,d}_{loc} ,$ the set of all
locally square integrable purely discontinuous martingales,
with predictable quadratic covariation process
\begin{equation}
<M>_t\; =\int_0^t g^2 (X_s) \bar f (X_s )ds 
\end{equation}
where
$$ \frac{< M>_t}{t} \to \pi ( g^2 \bar f ) $$
almost surely, as $t \to \infty.$ By the martingale convergence theorem, see \textit{e.g.}\ Jacod-Shiryaev (2003) \cite{js}, chapter VIII, Corollary 3.24 , $ t^{-1/2} M_t $ converges in law to a normal distribution. As a consequence, $M_t/t \to 0$ almost surely. 

We now treat the second term in \eqref{eq:martingaleplusbiaisconv}. By the ergodic theorem for integrable additive functionals, 
$$ \frac{1}{t} \int_0^t\int_{\R_+^N} g (x) \hat \mu ( ds, dx) = \frac{1}{t} \int_0^t g (X_s) \bar f( X_s) ds \to \pi ( \bar f g ),$$
and this finishes the proof.
\end{proof}

{\bf Exchangeability of the invariant measure}
We denote by $X^i : \R_+^N \to \R_+, ( x^1 , \ldots , x^N ) \mapsto x^i $ the $i-$th coordinate map.

\begin{prop}
For all $ 1 \le i \le N, $ ${\mathcal L}_\pi ( X^i ) = {\mathcal L}_\pi ( X^1).$
\end{prop}

\begin{proof}
Fix an initial configuration $x = (x^1, x^1 , \ldots , x^1) \in [0, K]^N$ consisting of $N$ particles which are all in the same position. Let $ \check g : \R_+ \to \R $ be a bounded test function and introduce $g ( x) := \check g (x^1 ) ,$ \textit{i.e.}\ $g$  depends only on the first coordinate.
By the ergodic theorem, 
$$ \frac1t \int_0^t g ( X_s) ds \to \int_\R \check g d {\mathcal L}_\pi ( X^1) $$
$P_x-$almost surely.

Now, introduce the system $ Y_t = (Y_t^1, \ldots , Y_t^N ) $ given by $ Y_t^k = X_t^k $ for all $ k \neq 1, i $ and $ Y_t^1 = X_t^i , $ $ Y_t^i = X_t^1 .$ Since the generator of $X$ is invariant under permutations, $ (Y_t)_{t \geq 0 } \stackrel{\mathcal L}{=} (X_t)_{t \geq 0}.$ In particular, 
$$ \int_\R \check g d {\mathcal L}_\pi ( X^1)  = \lim_{t \to \infty } \frac1t \int_0^t g ( X_s) ds = \lim_{ t \to \infty } \frac1t \int_0^t g ( Y_s) ds  $$
On the other hand,
$$ \lim_{ t \to \infty } \frac1t \int_0^t g ( Y_s) ds = \lim_{ t \to \infty } \frac1t \int_0^t \check g ( X^i_s) ds = \int_\R \check g d {\mathcal L}_\pi ( X^i),$$
and this finishes the proof.
\end{proof}
We are now going to study the support properties of the invariant measure of a single neuron. For that sake define for all $x \in [0,K], \; b(x):=\lambda (x-m)$ and recall that  $\gamma_t (x^i) \in \R_+$ denotes the solution of $ d \gamma_t (x^i) =-  b( \gamma_t (x^i)) dt ,$  given by 
$$ \gamma_t (x^i) = e^{- \lambda t }x^i + ( 1 - e^{- \lambda t })m.$$ 
Moreover, for $ x = (x^1 , \ldots , x^N), $ $ \gamma_t ( x ) = ( \gamma_t (x^1 ), \ldots , \gamma_t ( x^N) ) .$ Finally, let 
$$K(x, dy) = \sum_{i=1}^N f( x^i ) \delta_{  \Delta^i ( x) } (dy) , \quad H_f^x (t) = e^{ - \int_0^t \bar f ( \gamma_s ( x) ) ds } ,$$
where $ \bar f ( \gamma_s ( x) ) = \sum_{ i = 1 }^N f ( \gamma_s ( x^i ) ) $ and where $ \Delta^i ( x) $ was defined in \eqref{eq:delta} before.

We will use the change of variable, for a fixed value of $y,$
\begin{equation}\label{eq:covz}
z = \gamma_t ( y^1) , dz = - b (z) dt = - \lambda ( z-m) dt 
\end{equation} 
and denote by $ \kappa_y (z) $ the inverse function of $t \to \gamma_t ( y^1) .$

These definitions permit to obtain an expression of $\pi_1$.

\begin{prop} For all $z \in S_{d,k},$ we have
\begin{equation}\label{eq:pi1dens}
\pi_1 (z) =  \int \pi ( dx) \int K(x, dy ) H_f^y (\kappa_y   ( z) ) \frac{1_{I_y^m}(z)}{|b ( z)| } .
\end{equation}
Here the notation $I_y^m$ denotes either $]y,m[$ if $y<m$ or $]m,y[$ if $m<y.$
\end{prop}

\begin{proof}
We have, by Proposition \ref{prop:3},  
$$
\pi ( G) =  E_\pi ( G ( X_t )) = E_\pi ( \bar f (X_t)  G ( X_t) \frac{1}{\bar f ( X_t) } )=  \pi ( \bar f ) E_{\pi^Z}  \left( G ( Z_n )  \frac{1}{\bar f (Z_n) }\right) . 
$$
We use that $ \pi^Z = {\mathcal L}(X_{T_2-} | X_{T_1-} \sim \pi^Z ) . $ 
Then we obtain
\begin{eqnarray}\label{eq:piG}
\pi (G) &=& \pi ( \bar f ) \pi^Z \left(  \frac{G ( Z_n  ) }{\bar f (Z_n) }\right) \nonumber \\
&=&  \int_{\R_+^N} \bar f (x) \pi (dx) \sum_{i=1}^N \frac{f ( x^i ) }{\bar f (x) } \int_0^\infty \bar f ( \gamma_t (  \Delta_i (x) )) \nonumber  \\
&&\quad e^{- \int_0^t \bar f ( \gamma_s (  \Delta_i (x) )) ds } G ( \gamma_t (  \Delta_i ( x) ) ) \frac{1 }{\bar f ( \gamma_t (  \Delta_i ( x) ))} dt \nonumber \\
&=&  \int_{\R_+^N} \pi ( dx) \int K(x, dy)  \int_0^\infty H_f^y (t)  G( \gamma_t (y) )  dt . 
\end{eqnarray}
Now, let $g  \in C^\infty_c (S_{d , \beta }) $ be a smooth test function having compact support in $ S_{d , \beta }.$ Using \eqref{eq:piG}, we obtain
\begin{equation}\label{eq:pi1g}
\pi_1 ( g) =\int_{\R_+^N} \pi ( dx) \int K(x, dy)  \int_0^\infty H_f^y (t)  g( \gamma_t (y^1) ) dt . 
\end{equation}
Then, with the change of variable announced in \eqref{eq:covz}, we can rewrite \eqref{eq:pi1g} in the following way:
$$
\pi_1 (z) =  \int \pi ( dx) \int K(x, dy ) H_f^y (\kappa_y   ( z) ) \frac{1_{I_y^m}(z)}{|b ( z)| } .
$$
\end{proof}

\subsection{Support of the invariant measure}

\begin{prop}\label{prop:suptrans}
For all $y \in ]0,K[,$ all $\delta > 0,$ we have
$$
\inf_{f \in H( \beta , F, L, f_{min} )} \; \pi_1 \left(B_\delta(y) \right)>0  .
$$
\end{prop}

\begin{proof}
Fix $y \in ]0,K[$ and let $k \in \N$ and $s \in [0,\frac{1}{N}[$ be such that $y=\frac{k}{N}+s.$

We define the time $t_s$ such that $m \left( 1-e^{-\lambda t_s} \right) = s$ and consider, for a fixed $\varepsilon > 0,$ the following events:

$$
 A_y = \left\{ \frac{\ge}{2}  < T_1 <  \ge ; t_s +(i-1)\ge < T_i < t_s + (i-1) \ge + \frac{\ge}{2}  \; \forall i = 2 , \ldots, k +1 \right\} 
$$
and
$$
S_y= \{ I_1 =1, I_2 \neq 1 , \ldots , I_{k+1} \neq 1  \}.
$$
The idea of the proof is that the event $A_y \cap S_y$ leads the neuron 1 to a position close to $y$ after a time $t_y:=t_s + k \ge:$

At time $T_1$ the neuron 1 jumps so that its position is reset to $0,$ the time $t_s$ is defined such that at time $T_{2-}$ the position of neuron 1 is close to $s,$ then in an interval of time short enough for the deterministic drift to be insignificant, we impose that the other neurons jump $k$ times  so that at time $T_{k+1},$ the position of neuron 1 is indeed close to $y.$

In other words we can use similar arguments to the ones used in the proof of Lemma \ref{lem:positionbresiliens} to obtain that, for all $x \in [0,K]^N,$ if $ X_0 = x,$ then on the event $ A_y \cap S_y , $ we have $X_{T_{k+1}}^1=y+O(\ge),$ and we can choose $\ge$ such that $X_{T_{k+1}}^1 \in B_\delta(y).$

Now we have to prove that 
$$\inf_{f \in H( \beta , F, L, f_{min} )}  \; \inf_{ x \in [0,K]^N } P_x (A_y \cap S_y ) > 0 ,$$
which can be done as in the proof of proposition \ref{prop:control}.

Finally, integrating this result against the measure $\pi$ gives us the conclusion of the proof.
\end{proof}

We can now obtain \eqref{eq:noloss} as corollary of the following Proposition.

\begin{prop}\label{prop:imlb}
We have that
\begin{equation}\label{eq:imlb}
r^* := \inf_{a \in S_{d , \beta }} \;  \inf_{f \in H( \beta , F, L, f_{min} )}  \; \pi_1 (a) > 0,
\end{equation}
and for all $x \in [0,K]^N,$and for all $r \le r^*, $ 
\begin{equation}\label{eq:imdlb}
\lim \inf_{t \to \infty } \inf_{f \in H( \beta , F, L, f_{min} ) } P^f_x (A_{t,r} ) = 1.
\end{equation}

\end{prop}

\begin{proof}
Recalling the construction of $ \pi_1$ in (\ref{eq:pi1dens}), we have 
\begin{multline*}
\pi_1(a)= \int \pi ( dx) \int K(x, dy ) H_f^y (\kappa_y   ( a) ) \frac{1_{I_y^m}(a)}{|b ( a)| } \\
= \int \pi ( dx) \int \left( \sum_{i=1}^N f(x^i) \delta_{\Delta^i(x)}(dy) \right) \exp \left( -\int_0^{\kappa_y(a)} \bar f (\gamma_u(x))du \right) \frac{1_{I_y^m}(a)}{|b ( a)| }.
\end{multline*}
To obtain a lower bound uniform in $f$ of this expression we use again the bounds of the class of function $H( \beta , F, L, f_{min} ):$
$$
\forall f \in H( \beta , F, L, f_{min} ), \; f_{min}(x) \leq f(x) \leq F.
$$
Doing this, we will also need an upper bound for $\kappa_y(a).$ This is possible due to the term $1_{I_y^m}(a):$ since $y$ is such that $a \in I_y^m \cap S_{d , \beta },$ the flow starting from $y$ can reach $a$ in a finite time, even if we consider the worst cases where $y=0$ or $K.$

Thanks to Proposition  \ref{prop:suptrans}, we have $\pi_1 \left( \left\{ y: a \in I_y^m \right\} \right) >0$ implying that the integration of $1_{I_y^m}(a)$ against the measure $\pi_1(dy)$ is not 0. Finally, due to the definition of $S_{d , \beta },$  we have no problem to obtain this lower bound uniformly in $a \in S_{d , \beta },$ and this finishes the proof of \eqref{eq:imlb}.

\eqref{eq:imdlb} is obtained easily from \eqref{eq:imlb} thanks to the ergodic theorem: we have

$$
\lim_{t \to + \infty} \frac{1}{Nt} \int_0^t \int_\R Q_h (y-a) \eta(ds,dy) = \pi_1(a),
$$
(recall that $\int_\R Q(x)dx=1$), which concludes the proof.
\end{proof}

\section{proof of theorem \ref{theo:main}}\label{sec:proofmain}

\subsection{Convergence of the estimator}
We now study the speed of convergence of our estimator. First we have the following classical kernel approximation:

\begin{prop}\label{prop:KernApprox}
For any H\"older function $g$ of order $\beta=k + \alpha$ satisfying 
\begin{equation}\label{eq:holderg}
\sup_{w \neq w'} \big| g^{(k)} (w) - g^{(k)} (w') \big| \le C_g |w-w'|^\alpha 
\end{equation}  for some constant $C_g$ and for a kernel $Q$ as in Theorem \ref{theo:main}, we have:
$$
\left| \int_\R Q_h (y-a) g(y) dy - g(a) \right| \leq  \frac{C_g \parallel Q \parallel_{L^1} R^\beta}{k!} h^{\beta}, 
$$
where we recall that $R$ is the diameter of the support of $Q$ and where $\parallel Q \parallel_{L^1}$ denotes the $L^1$-norm of $Q.$
\end{prop}

\begin{proof}
Using the property $\int_\R Q(x)dx =1$ and the change of variable $x=\frac{y-a}{h},$ we obtain
$$
\int_\R Q_h (y-a) g(y) dy - g(a) = \int_\R Q(x) \left( g(a+xh) - g(a) \right) dx.
$$
Then, a Taylor-Lagrange expansion of the function $g$ gives us
$$
\int_\R Q_h (y-a) g(y) dy - g(a) = \int_\R Q(x) \left( \sum_{l=1}^k \frac{g^{(l)}(a)}{l!} (xh)^l + \frac{g^{(k)}(z)-g^{(k)}(a)}{k!}(xh)^k \right) dx,
$$
for some $z \in ]a,a+xh[ \cup ]a + xh , a[ .$ By the assumptions of Theorem \ref{theo:main}, $\int_{\R}  Q(y) y^j dy = 0 $ for all $ 1 \le j \le k.$ Then condition (\ref{eq:holderg}) allows to conclude.
\end{proof}

Fix $a \in S_{d , \beta } $ and define, for all  $t \in \R^+, \; \tilde \mu = \mu - \hat \mu$ the centered jump measure.

\begin{prop}\label{prop:ernum}
Under the conditions of Theorem \ref{theo:main}, there exists a constant $C_1$  depending only on $\beta ,  F, L, N, f_{min}$ and $Q,$ such that for all $f \in H( \beta , F, L, f_{min} ), $ for all $x \in [0,K]$ and for a bandwidth of the form $h  = h_t = t^{- \alpha } $ for some $ 0 < \alpha < 1,$ 
\begin{equation}\label{eq:ernumx}
E_x \left[ \left( \frac{1}{Nt} \int_{ [0, t]} \int_\R Q_h (y-a) \tilde  \mu (ds, dy)  \right)^2 \right] \leq \frac{C_1}{ht}.
\end{equation}
\end{prop}

\begin{proof}
We start working under the invariant regime in the first part of the proof, \textit{i.e.}\ we will work under $ E_\pi .$   In a second time we will use Theorem \ref{theo:harrisok} to obtain the result for any starting point $x \in [0,K]^N.$ 

We use the properties of the compensator $\hat \mu_t$ and its explicit expression to write
\begin{multline*}
E_{\pi} \left[ \left( \frac{1}{Nt} \int_{[0, t ]}\int_\R Q_h (y-a) \tilde \mu(ds, dy)  \right)^2 \right] 
 = \frac{1}{(Nt)^2} E_{\pi} \left[ \int_0^t \int_\R \left( Q_h (y-a) \right)^2 \hat \mu(ds, dy) \right] \\ 
 = \frac{1}{(Nt)^2} E_{\pi} \left[ \int_0^t \int_\R \left( Q_h (y-a) \right)^2 f(y) \eta(ds,dy)  \right].
\end{multline*}

Now, since we are in the invariant regime, we can use 
the density of the invariant measure of a single particle (recall Theorem \ref{theo:invmeasure}) to obtain
$$
E_{\pi} \left[ \left( \frac{1}{Nt} \int_{[0, t ]} \int_\R Q_h (y-a) \tilde \mu(ds, dy)  \right)^2 \right] = \frac{1}{Nt}  \int_\R \left( Q_h (y-a) \right)^2 f(y) \pi_1 (y)dy. 
$$
Our aim is to obtain a control of $ \int_\R  h f(y) ( Q_h(y -a))^2 \pi_1(y)dy$ independently of $h.$ To do this we use the change of variable $x=\frac{y-a}{h}$ and write
$$
E_{\pi} \left[ \left( \frac{1}{Nt} \int_{[0, t ]} \int_\R Q_h (y-a) \tilde \mu(ds, dy)  \right)^2 \right] \\
= \frac{1}{Nht}  \int_\R Q^2 (x) f(a+xh) \pi_1 (a+xh)dx .
$$
This yields 
\begin{equation}\label{eq:numinvreg}
E_{\pi} \left[ \left( \frac{1}{Nt} \int_{[0, t]}\int_\R Q_h (y-a) \tilde \mu(ds, dy)  \right)^2 \right] \leq \frac{F}{Nht} \| Q \|_{L^2 }^2 \sup_{ x \in S_{d/2 , k }} \pi_1 ( x).
\end{equation}
This result holds in stationary regime, but thanks to the exponential speed of convergence of Theorem \ref{theo:harrisok}, we can obtain it for any starting point $x \in [0,K]^N$ as we are going to show now. For that sake we fix the bandwidth $h$ in function of $t$ so that this speed of convergence depends only on $t.$ For the moment, we will assume that $h$ is of the form
\begin{equation}\label{eq:htmalpha}
h_t:=t^{-\alpha}
\end{equation}
for some constant $\alpha \in ]0,1[.$ As in the beginning of the proof, we can write
$$
E_x \left[ \left( \frac{1}{Nt} \int_{[0, t ]}\int_\R Q_h (y-a) \tilde  \mu(ds, dy)  \right)^2 \right] = \frac{1}{(Nt)^2} E_x \left[ \int_0^t \int_\R \left( Q_h (y-a) \right)^2 f(y) \eta(ds,dy)  \right].
$$
Now, we have the following decomposition
\begin{multline*}
E_x \left[ \left( \frac{1}{Nt} \int_{[0, t ]}\int_\R Q_h (y-a) \tilde  \mu(ds, dy)  \right)^2 \right] \\
= \frac{1}{(N t)^2} E_x \left[ \sum_{i=1}^N \int_0^t \left( Q_h (X_s^i-a) \right)^2 f(X_s^i) ds  \right] - \frac{1}{(N t)^2} E_{\pi} \left[ \sum_{i=1}^N \int_0^t \left( Q_h (X_s^i-a) \right)^2 f(X_s^i) ds  \right] \\
 + \frac{1}{(N t)^2} E_{\pi} \left[ \sum_{i=1}^N \int_0^t \left( Q_h (X_s^i-a) \right)^2 f(X_s^i) ds  \right].
\end{multline*}
The last term is controlled by \eqref{eq:numinvreg}. We will deal with the difference in the second line using Theorem \ref{theo:harrisok} as follows: for all $p \in ]0,1-\alpha[,$ we have
\begin{multline*}
\frac{1}{(N t)^2} E_x \left[ \sum_{i=1}^N \int_0^t \left( Q_h (X_s^i-a) \right)^2 f(X_s^i) ds  \right] - \frac{1}{(N t)^2} E_{\pi} \left[ \sum_{i=1}^N \int_0^t \left( Q_h (X_s^i-a) \right)^2 f(X_s^i) ds  \right] \\
=  \frac{1}{(N t)^2} E_x \left[ \sum_{i=1}^N \int_0^{t^p} \left( Q_h (X_s^i-a) \right)^2 f(X_s^i) ds  \right] - \frac{1}{(N t)^2} E_{\pi} \left[ \sum_{i=1}^N \int_0^{t^p} \left( Q_h (X_s^i-a) \right)^2 f(X_s^i) ds  \right] \\
+ \frac{1}{(N t)^2} \sum_{i=1}^N \int_{t^p}^t \Big( E_x \left[ \left( Q_h (X_s^i-a) \right)^2 f(X_s^i) \right] -  E_{\pi} \left[ \left( Q_h (X_s^i-a) \right)^2 f(X_s^i) \right] \Big) ds.
\end{multline*}
To conclude, we use the upper bounds $\parallel Q \parallel_\infty$ and $F$ for $Q$ and $f$ to control the second line and we use Theorem \ref{theo:harrisok} to control the last term. As a consequence, 
\begin{multline*}
\left| \frac{1}{(N t)^2} E_x \left[ \sum_{i=1}^N \int_0^t \left( Q_h (X_s^i-a) \right)^2 f(X_s^i) ds  \right] - \frac{1}{(N t)^2} E_{\pi} \left[ \sum_{i=1}^N \int_0^t \left( Q_h (X_s^i-a) \right)^2 f(X_s^i) ds  \right] \right| \\
\leq \frac{F \parallel Q \parallel_\infty^2}{Nh^2 t^2} \left( 2 t^p + C \int_{t^p}^t \kappa^{-s} ds \right) = \frac{F \parallel Q \parallel_\infty^2}{Nh^2 t^2} \left( 2 t^p + C \frac{\kappa^{-t^p}-\kappa^{-t}}{\ln (\kappa)} \right) = \frac{1}{ht} {\mathcal O} \left( \frac{t^p}{ht} \right) .
\end{multline*}
Now recall that $h=h_t=t^{-\alpha}$ by (\ref{eq:htmalpha}) and that $p \in ]0,1-\alpha[.$ Thus
$$
\frac{1}{(N t)^2} E_x \left[ \sum_{i=1}^N \int_0^t \left( Q_h (X_s^i-a) \right)^2 f(X_s^i) ds  \right] - \frac{1}{(N t)^2} E_{\pi} \left[ \sum_{i=1}^N \int_0^t \left( Q_h (X_s^i-a) \right)^2 f(X_s^i) ds  \right] = o \left( \frac{1}{ht} \right),
$$
which allows to conclude.
\end{proof}

Proposition \ref{prop:ernum} will help us to control the numerator of our estimator. We want to establish the same kind of result for the denominator and this leads to the following proposition:

\begin{prop}\label{prop:erden}
For all $a \in S_{d , \beta },$ define
\begin{equation}\label{eq:defQtilde}
\tilde Q_{h,f}(y):=Q_h ( y - a ) \Big( f(y)-f(a) \Big)  - \pi_1 \Big( Q_h ( \cdot - a ) \Big( f(\cdot)-f(a) \Big) \Big).
\end{equation}
Under the conditions of Theorem \ref{theo:main}, there exists a constant $C_2$  depending only on $\beta , F, L, N, f_{min}$ and $Q,$ such that for all $f \in H( \beta , F, L, f_{min} ), $ for all $x \in [0,K]$ and for a bandwidth of the form $h  = h_t = t^{- \alpha } $ for some $ 0 < \alpha < 1,$
\begin{equation}\label{eq:erdenx}
E_x \left[ \left( \frac{1}{Nt} \int_0^t \int_\R \tilde Q_{h,f}(y) \eta(ds,dy)  \right)^2 \right] \leq \frac{C_2}{t} h^{2 \left( 1\wedge \beta \right) -1}.
\end{equation}
\end{prop}

\begin{proof} 
As in the preceding proof we start by working in the stationary regime, \textit{i.e.}\ under $E_\pi .$ 
\begin{multline}\label{eq:errorstartden}
 E_{\pi} \left[ \left( \frac{1}{Nt} \int_0^t \int_\R \tilde Q_{h,f}(y) \eta(ds,dy)  \right)^2 \right]
\\
\leq \frac{2}{(Nt)^2} E_{\pi} \left[ \int_0^t  \int_\R \Big| \tilde Q_{h,f}(x) \Big| \eta(ds,dx) \Big| E_{\pi} \left( \int_s^t  \int_\R \tilde Q_{h,f}(y) \eta(du,dy) \Big| {\mathcal F}_s \right) \Big| \right].
\end{multline}
We deal with the conditional expectation using the Markov property and write
\begin{multline*}
E_{\pi} \left( \int_s^t  \int_\R \tilde Q_{h,f}(y) \eta(du,dy) \Big| {\mathcal F}_s \right)
\\
= E_{X_s} \left( \int_0^{t-s} \sum_{i=1}^N \tilde Q_{h,f}(X_u^i) du\right) = \int_0^{t-s}  \sum_{i=1}^N  E_{X_s} \left( \tilde Q_{h,f}(X_u^i) \right) du.
\end{multline*}
Now going back to the definition of $\tilde Q_{h,f},$ we can use Theorem \ref{theo:harrisok} and write
\begin{multline*}
E_{X_s} \left( \tilde Q_{h,f}(X_u^i) \right) = E_{X_s} \Big( Q_h ( X_u^i - a ) \left( f(X_u^i) - f(a) \right) \Big)  - \pi_1 \Big( Q_h ( \cdot - a ) \left( f(\cdot ) - f(a) \right) \Big) \\
\leq \frac{C}{h} (F\vee L)(Rh)^{1\wedge \beta} \parallel Q \parallel_\infty \kappa^{-u},
\end{multline*}
due to the assumption (\ref{eq:Hspace}) on the H\"older space containing $f.$ (Recall that $R$ is the diameter of the support of $Q.$) The integrability of the function $ u \to \kappa^{-u} $allows to deduce from this that
$$
\Big| E_{\pi} \left( \int_s^t  \int_\R \tilde Q_h (y-a) \eta(du,dy) \Big| {\mathcal F}_s \right) \Big| \leq \frac{N \tilde C}{h} (F\vee L)(Rh)^{1\wedge \beta} \parallel Q \parallel_\infty
$$
for some constant ${\tilde C}.$ Taking this result into account in (\ref{eq:errorstartden}), we obtain
$$
  E_{\pi} \left[ \left( \frac{1}{Nt} \int_0^t \int_\R \tilde Q_{h,f}(y) \eta(ds,dy)  \right)^2 \right] \leq \frac{2{\tilde C}}{Nht^2}  (F\vee L)(Rh)^{1\wedge \beta} \parallel Q \parallel_\infty E_{\pi} \left[\int_0^t \int_\R \Big| \tilde Q_{h,f}(x) \Big| \eta(ds,dx) \right].
$$
The end of the proof is similar to the one of Proposition \ref{prop:ernum}: the fact that we are in the invariant regime allows to use the density of the invariant measure of a single particle and its control given by Theorem \ref{theo:invmeasure}. Then we use the same change of variable $x=\frac{y-a}{h}$ to obtain
$$
E_{\pi} \left[ \left( \frac{1}{Nt} \int_0^t \int_\R \tilde Q_{h,f}(y) \eta(ds,dy)  \right)^2 \right] \leq \frac{4 {\tilde C} }{ht} (F\vee L)^2(Rh)^{2 \left( 1\wedge \beta \right) } \parallel Q \parallel_\infty  \parallel Q \parallel_{L^1} \sup_{ x \in S_{d/2 , k }} \pi_1 ( x).
$$
This result is established under the invariant regime, but we are able to extend it to any starting point $x \in [0,K]^N,$ using the same trick as the one in the proof of Proposition \ref{prop:ernum}. This finishes the proof.
\end{proof}

\subsection{Proof of Theorem \ref{theo:main}, $(i)$}

Introducing
$$ D^{t,h} = \frac{1}{Nt} \int_\R \frac1h Q\left( \frac{ y-a}{h} \right) \eta_t (dy ),$$
we have
\begin{multline*}
 D^{t,h} ( \hat f_{t,h} ( a) - f (a) ) = \frac{1}{Nt} \int_{[0, t ]}\int_\R Q_h (y-a) \mu( ds, dy) -f(a)D^{t,h} \\
 = \frac{1}{Nt} \int_{[0, t ]}\int_\R Q_h (y-a) \tilde \mu( ds, dy) + \frac{1}{Nt} \int_0^t \int_\R \frac1h Q\left( \frac{ y-a}{h} \right) \left( f(y) - f(a) \right) \eta (ds, dy ).
\end{multline*}
With the definition of $\tilde Q_{h,f}$ in (\ref{eq:defQtilde}), we have the following decomposition:
\begin{multline}\label{eq:erdecomp}
 D^{t,h} ( \hat f_{t,h} ( a) - f (a) ) \\
 = \frac{1}{Nt} \int_{[0, t ]}\int_\R Q_h (y-a) \tilde \mu(ds, dy) + \frac{1}{Nt} \int_0^t \int_\R \tilde Q_{h,f}(y)  \eta (ds, dy ) + \pi_1 \Big( Q_h ( \cdot - a ) \left( f(\cdot ) -  f(a) \right) \Big)   .
\end{multline}
The first two terms of the previous sum are  controlled respectively by Propositions \ref{prop:ernum} and \ref{prop:erden}. We deal with the third term using Proposition \ref{prop:KernApprox} as follows:

\begin{multline*}
\pi_1 \Big( Q_h ( \cdot - a ) \left( f(\cdot ) -  f(a) \right) \Big) = \int_\R  Q_h ( y - a ) \left( f(y ) -  f(a) \right) \pi_1 (y) dy \\
= \left( \int_\R  Q_h ( y - a ) \left( f(y)  \pi_1 (y) -  f(a) \pi_1 (a) \right) dy \right) + f(a) \left( \int_\R  Q_h ( y - a ) \left( \pi_1 (a) -   \pi_1 (y) \right) dy \right).
\end{multline*}

Both functions $\pi_1$ and $f \pi_1$ are H\"older of order $\beta$ (recall Theorem \ref{theo:invmeasure}) and we can apply Proposition \ref{prop:KernApprox} to each of the last two terms, using the upper bound $F$ for $f(a).$

Putting all together in (\ref{eq:erdecomp}), we have
\begin{equation}\label{eq:ersum}
\parallel D^{t,h} \left( \hat f_{t,h} ( a) - f (a)  \right) \parallel_{L^2(P^f_x)} \leq \sqrt{\frac{C_1}{ht}} + \sqrt{\frac{C_2}{ht}}h^{1\wedge \beta} + C_3 h^\beta ,
\end{equation}
with constants $C_1 , C_2$ and $C_3$ depending only on $\beta , F, L, f_{min}$ and $Q.$
As in the proof of Proposition \ref{prop:erden}, we will fix the dependence in $t$ of $h$ putting $h_t:=t^{-\alpha}$ and choosing $\alpha \in ]0,1[$ to obtain an optimal speed of convergence.

This leads to the choice $\alpha := \frac{1}{2\beta + 1}$ and $h=h_t=t^{-\frac{1}{2\beta + 1}}$ which gives us 
$$
\parallel D^{t,h_t} \left( \hat f_{t,h_t} ( a) - f (a) \right) \parallel_{L^2(P_x^f)} \leq C(\beta , F, L, f_{min}, Q) t^{-\frac{\beta}{2 \beta +1}}.
$$
To finish the proof of Theorem \ref{theo:main}, we have to work conditionally on the event $A_{t,r}$, for $ r \le r^*, $ on which we have $D^{t,h} \geq r .$ 
\begin{multline*}
E_x \left[ \left( \hat f_{t,h_t} ( a) - f (a) \right)^2 \Big| A_{t,r} \right] = \frac{1}{P_x \left( A_{t,r} \right) }  E_x \left[ \left( \hat f_{t,h_t} ( a) - f (a) \right)^2 1_{A_{t,r}} \right] \\ 
\leq \frac{1}{r^2 P_x \left( A_{t,r} \right) } \parallel D^{t,h_t} \left( \hat f_{t,h_t} ( a) - f (a) \right) \parallel_{L^2(P_x^f)}^2 \leq \frac{C(\beta , F, L,f_{min}, Q)^2 t^{-\frac{2 \beta}{2 \beta +1}}}{r^2 P_x \left( A_{t,r} \right) },
\end{multline*}
and the conclusion follows thanks to (\ref{eq:imdlb}). \hfill $\qed $

\subsection{Proof of Theorem \ref{theo:main} $(ii)$:} 

The proof relies on the martingale convergence theorem given in Corollary 3.24 of \cite{js} chapter VIII. We use the following decomposition
\begin{equation}\label{eq:decomptcl}
D^{t,h} ( \hat f_{t,h} ( a) - f (a) ) = \frac{1}{N\sqrt{th}} M^{t,h} + \frac{1}{Nt} \int_0^t \int_\R Q_h (y-a) \left( f(y)-f(a) \right) \eta(du,dy ),
\end{equation}
where
$$
M^{t,h}:= \frac{1}{\sqrt{th}} \int_{[0, t ]}\int_\R Q \left(\frac{y-a}{h} \right) \tilde \mu(ds, dy).
$$
We define for all $t \in \R_+$
$$
(M^t)_s := \frac{1}{\sqrt{th}} \int_{[0, ts]} \int_\R Q \left(\frac{y-a}{h} \right) \tilde \mu(du, dy)
$$
and show that the Assumption 3.23 of \cite{js} chapter VIII is satisfied for this sequence of processes. Therefore, we have to study, for all $\varepsilon>0$ and all $s \in \R_+,$ the limit of 
$$
\frac{1}{th} \int_0^{ts} \sum_{i=1}^N f \left( X_u^i \right) Q^2 \left( \frac{X_u^i-a}{h} \right) 1_{ \{ \frac{1}{\sqrt{th}} Q \left( \frac{X_u^i-a}{h} \right)> \varepsilon \} } du
$$

as $t$ goes to $+\infty.$ Since $Q$ is bounded and $\lim_{t \to +\infty} th_t = +\infty ,$ there exists $t_0$ such that for all $t>t_0,$ $1_{ \{ \frac{1}{\sqrt{th}} Q \left( \frac{X_u^i-a}{h} \right)> \varepsilon \} }=0.$ Consequently, the above limit is $0$ and Assumption 3.23 of \cite{js} chapter VIII is indeed satisfied.

Moreover, 
$$
\left< M^t, M^t \right>_s = \frac{1}{th} \int_0^{ts} \int_\R Q^2 \left( \frac{y-a}{h} \right) f(y) \eta(du,dy).
$$
Since our process is positive Harris recurrent, by the ergodic theorem, we have the following proposition.

\begin{prop}
$\left< M^t, M^t \right>_s$ converges in $P_x$-Probability as $t$ goes to $+\infty$ to 
$$
Nsf(a) \pi_1(a) \int Q^2(x)dx \; {\mbox a.s.}
$$

\end{prop}

\begin{proof}
Since our process is positive Harris recurrent, $f$ being continuous and $Q$ with compact support, we have
$$
\lim_{t \to + \infty} E_x \left[ \left( \frac{1}{th} \int_0^{ts} \int_\R Q^2 \left( \frac{y-a}{h} \right) f(y) \eta(du,dy) - \frac{N}{th} \int_0^{ts} \int_\R Q^2 \left( \frac{y-a}{h} \right) f(y) \pi_1(y)dy \right)^2 \right] =0.
$$
Then the result is obtained by continuity of $\pi_1$ and $f$ on $S_{d,k}.$
\end{proof}

Consequently, Corollary 3.24 of \cite{js} chapter VIII with $s=1$ gives us the weak convergence of $M^{t,h}$ to 
${\mathcal N} \left( 0, Nf(a) \pi_1(a) \int Q^2(x)dx \right).$

We deal with the second term of (\ref{eq:decomptcl}) as in the previous subsection and obtain 
$$
\left\Vert \frac{1}{Nt} \int_0^t \int_0^K Q_h (y-a) \left( f(y)-f(a) \right) \eta(du,dy ) \right\Vert_{L^2(P^f_x)} \leq \sqrt{\frac{C_2}{ht}}h^{1\wedge \beta} + C_3 h^\beta .
$$
Therefore, when $t$ goes to $+ \infty,$ (\ref{eq:decomptcl}) gives us the following weak convergence:
$$
\sqrt{th_t}D^{t,h_t} ( \hat f_{t,h_t} ( a) - f (a) ) \longrightarrow {\mathcal N} \left( 0, \frac{ f(a) \pi_1(a)}{N} \int Q^2(x)dx \right),
$$
since $ h_t = o ( t^{ - 1 /(1 + 2 \beta ) } ).$

Finally, we deal with the additive functional $D^{t,h_t}$ using the ergodic theorem.
Recall that 
$$ D^{t,h} = \frac{1}{Nt} \int_0^t \int_\R \frac1h Q\left( \frac{ y-a}{h} \right) \eta(ds,dy ).$$
Thanks to \eqref{eq:imlb}, $\pi_1(a) >0,$ and the ergodic theorem gives us the almost sure convergence to $\pi_1(a)$ (since $\int Q(x)dx=1$), which allows us to conclude.
\hfill $\qed $

\section{Proof of Theorem \ref{theo:lowerbound}}\label{sec:optimal}
The proof of Theorem \ref{theo:lowerbound} follows closely the proof of Theorem 8 of Hoffmann and Olivier (2015) \cite{hoffmann-olivier}, going back to similar ideas developed in \cite{hhl}.
Let $ h_t = t^{ - \frac{1}{2 \beta + 1 } }$ and fix any test rate function $f_0 \in H( \beta, F - \delta , L- \delta, f_{min} ) ,$ for some fixed $ \delta  \in ] 0, F \wedge L[ .$ Then, as in \cite{hoffmann-olivier}, we define a perturbation $f_t $ of $f_0$ by 
$$ f_t ( x ) = f_0 ( x) + b h_t^{\beta +1} \chi_{h_t} ( x - a ) , $$
where $ b > 0 $ is a positive constant, $\chi \in C_c ( \R_+, \R_+) $ is a positive kernel function of compact support included in $[- 1, 1 ] $ such that $ \chi ( 0 ) = 1 , $ $ \chi ( x) \le 1 $ for all $x$ and
\begin{equation}\label{eq:chiht}
\chi_{h_t} ( x) = \frac{1}{h_t} \chi ( \frac{x}{h_t} ) .
\end{equation}

 Notice that the first $l$ derivatives of $ \chi_{h_t} $ are of order $ h_t^{ - (l+1)} ,$ therefore the factor $ h_t^{\beta +1}$ implies that $f_t \in H( \beta , F, L, f_{min} ) ,$ if we choose $b  $ sufficiently small. An important point in the above choice of $f_t$ is that 
\begin{equation}\label{eq:auchgut}
 f_t ( a) - f_0 ( a) = b h_t^\beta = b t^{ - \frac{\beta }{2 \beta +1}} ,
\end{equation} since $\chi ( 0 ) = 1.$

In the following, we shall write shortly $ \P_0 := ( P_x^{f_0})_{| {\mathcal F_t}} $ and $ \P_t := (P_x^{f_t})_{| {\mathcal F_t}} $ for the associated probability measures in restriction to $ {\mathcal F}_t.$ The following lower bound is by now classical. For any fixed constant $C > 0,$ using Markov's inequality and denoting by $L_t^{f_t/f_0} =\frac{ d \P_0}{d \P_t }  $ the likelihood ratio of $ \P_0 $ with respect to $\P_t, $ on ${\mathcal F}_t,$    
\begin{eqnarray*}
&&\sup_{ f \in H( \beta , F, L, f_{min} ) } t^{ \frac{2\beta}{1 + 2 \beta } }  E_x^f [ | \hat f_t ( a) - f(a) |^2 ]  \\
&& \geq t^{ \frac{2\beta}{1 + 2 \beta } }  \left[ \frac12 \E_0 [ | \hat f_t ( a) - f_0 ( a) |^2 ]  + \frac12 \E_t [ | \hat f_t ( a) - f_t ( a) |^2 ] \right] \\
&&\geq \frac{C^2 }{2} \left[ \P_0 \left( t^{ \frac{\beta}{1 + 2 \beta } }| \hat f_t ( a) - f_0 ( a) | \geq C \right) +  \P_t \left( t^{ \frac{\beta}{1 + 2 \beta } }| \hat f_t ( a) - f_t ( a) | \geq C \right) \right] \\
&& = \frac{C^2 }{2} \left[ \P_0 \left( t^{ \frac{\beta}{1 + 2 \beta } }| \hat f_t ( a) - f_0 ( a) | \geq C \right) +  \E_0 \left( L_t^{f_t/f_0} 1_{\{ t^{ \frac{\beta}{1 + 2 \beta } }| \hat f_t ( a) - f_t ( a) | \geq C  \}}\right) \right] . 
%&&\geq 
%\frac{C^2 }{2} \left[ \E_0 \left( 1_{\{ t^{ \frac{\beta}{1 + 2 \beta } }| \hat f_t ( a) - f_0 ( a) | \geq C  \}}  + 1_{\{ t^{ \frac{\beta}{1 + 2 \beta } }| \hat f_t ( a) - f_t ( a) | \geq C  \}} 
%\right) - \| \P_0 - \P_t \|_{TV}  
%\right] .
\end{eqnarray*} 
Now, 
$$  t^{ \frac{\beta}{1 + 2 \beta } }[ | \hat f_t ( a) - f_0 ( a) |  + | \hat f_t ( a) - f_t ( a) |  ] \geq  t^{ \frac{\beta}{1 + 2 \beta } } | f_0(a)  - f_t ( a) | \geq b ,$$
which is due to \eqref{eq:auchgut}. As a consequence, if we choose
$
C = b /2 ,
$ then
$$ 1_{\{ t^{ \frac{\beta}{1 + 2 \beta } }| \hat f_t ( a) - f_0 ( a) | \geq C  \}}  + 1_{\{ t^{ \frac{\beta}{1 + 2 \beta } }| \hat f_t ( a) - f_t ( a) | \geq C  \}} \geq 1 ,$$
in particular, 
$$ 1_{\{ t^{ \frac{\beta}{1 + 2 \beta } }| \hat f_t ( a) - f_t ( a) | \geq C  \}}  \geq 1_{\{ t^{ \frac{\beta}{1 + 2 \beta } }| \hat f_t ( a) - f_0 ( a) | < C  \}}  .$$
We conclude that 
\begin{eqnarray*}
&&\sup_{ f \in H( \beta , F, L, f_{min} ) } t^{ \frac{2\beta}{1 + 2 \beta } }  E_x^f [ | \hat f_t ( a) - f(a) |^2 ]\\
&& \geq \frac{ b^2}{8} 
\E_0 \left[ 1_{\{ t^{ \frac{\beta}{1 + 2 \beta } }| \hat f_t ( a) - f_0 ( a) | \geq  \frac{b}{2}  \}}  + L_t^{ f_t/f_0} 1_{\{ t^{ \frac{\beta}{1 + 2 \beta } }| \hat f_t ( a) - f_0 ( a) | <  \frac{b}{2}  \}} \right] 
\\
&&\geq  \frac{ b^2}{8}  e^{- s} \P_0 ( L_t^{f_t/ f_0 } \geq e^{-s} ) ,
\end{eqnarray*}
for any $s > 0 .$ 
Therefore, in order to achieve the proof of Theorem \ref{theo:lowerbound}, it suffices to show that 
\begin{equation}\label{eq:sufficientcondlr}
 \lim \sup_{t \to \infty } \E_0 [ | \log L_t^{f_t/f_0 } | ] < \infty .
\end{equation} 

Indeed, we can deduce from \eqref{eq:sufficientcondlr} the following statements:
\begin{gather*}
\exists M, \forall t, \E_0 \left( | \log L_t^{f_t/f_0 } | \right) \leq M, \\
\exists M, \forall t, \P_0 \left(  \log L_t^{f_t/f_0 } < -2M \right) \leq \frac{1}{2}, \\
\exists s, \forall t, \P_0 \left(  \log L_t^{f_t/f_0 } \geq -s \right) \geq \frac{1}{2}.
\end{gather*}

Recall that by construction, $f_t \geq f_0.$ Moreover, since the support of $ \chi $ is included in $ [-1 , 1],$ $ f_t (y) \neq f_0 (y) $ implies $ y \in J_t := [ a - h_t, a+h_t ].$ Now, Theorem 3.5 of L\"ocherbach (2002) \cite{evaold}, applied to the particular case without branching, shows that $\P_0$ and $\P_t $ are equivalent on $ {\mathcal F}_t, $ with density 
\begin{equation}\label{eq:llr}
 \log L_t^{f_t/f_0} = \int_0^t \int_{J_t} \log ( \frac{f_t}{f_0} (y)) \mu (ds, dy ) - \int_{J_t} \left ( \frac{f_t}{f_0 } - 1 \right)(y)  f_0 (y) \eta_t (dy ) .
\end{equation}
We now proceed exactly as in \cite{hhl}, proof of Lemma 11. The $\P_0-$ martingale part within \eqref{eq:llr} is given by 
$$ \int_0^t \int_{J_t} \left ( \frac{f_t}{f_0 } - 1 \right) ( \mu -  \hat \mu^{f_0} ) (ds, dy ) ,$$
where $ \hat \mu^{f_0} (ds , dy ) = \sum_{i=1}^N f_0 ( X^i_s ) \delta_{X^i_s} ( dy ) ds $ is the $ \P_0-$compensator of $ \mu .$  Its angle bracket is 
\begin{multline*}
 b^2 h_t^{ 2 \beta +2} \int_{J_t} \left( \frac{ \chi^2_{h_t} ( y - a ) }{f_0 ( y ) } \right) \eta_t (dy )  \le \frac{ b^2 }{\inf_{y \in J_t} ( f_0 (y))  } t^{ - \frac{ 2 \beta}{ 2 \beta + 1 } } h_t^2 \int_{J_t} \chi^2_{h_t} ( y - a ) \eta_t (dy ) \\ 
 \le  \frac{ b^2 }{\inf_{y \in J_t} ( f_0 (y))  } t^{ - \frac{ 2 \beta}{ 2 \beta + 1 } } \eta_t ( J_t) =  \frac{ b^2 }{\inf_{y \in J_t} ( f_0 (y))  } t^{  \frac{ 1}{ 2 \beta + 1 } } \frac{1}{t} \eta_t ( J_t),
\end{multline*}
since $ \chi ( \cdot ) \le 1 ,$ by definition of $\chi_{h_t}$ (recall \eqref{eq:chiht}). All other terms in \eqref{eq:llr} are treated exactly as in \cite{hhl}. Therefore, it only remains to show that 
\begin{equation}\label{eq:finalouf}
\lim \sup_{ t \to \infty } \E_0 \left( \frac{1}{t h_t} \eta_t (J_t) \right) < \infty .  
\end{equation}
We apply once more Theorem \ref{theo:harrisok} and rewrite
\begin{multline*}
\E_0 ( \eta_t ( J_t) ) = \int_0^t E_x^{f_0} ( \bar 1_{J_t} (X_s) ) ds = 
\int_0^t E_x^{f_0} \left( \bar 1_{J_t} (X_s) - \pi^{f_0} ( \bar 1_{J_t} ) \right)  ) ds  + t \pi^{f_0} ( \bar 1_{J_t} ) 
\\
\le  C N \int_0^t \kappa^{ - s} ds + t \pi^{f_0} ( \bar 1_{J_t} )  
\le CN \frac{1}{\ln \kappa} + t \pi^{f_0} ( \bar 1_{J_t} ) \\
=  CN \frac{1}{\ln \kappa} + N t \int_{J_t} \pi_1^{f_0} (y ) dy ,
\end{multline*}
where $\bar 1_{J_t} (x):= \sum_{i=1}^N 1_{J_t}(x^i)$ for $x \in \R^N,$ and $\pi_1^{f_0} (y ) $ denotes the Lebesgue density of $ \pi_1^{f_0}, $ which exists on $J_t$ by choice of $a, $ for $t$ sufficiently large. Using the change of variables $ z = (y- a)/ h_t ,$ we obtain 
$$ \E_0 ( \eta_t ( J_t) ) \le CN \frac{1}{\ln \kappa} + N t h_t \int_{-1}^1 \pi_1^{f_0} ( a + h_t z) dz  \le CN \frac{1}{\ln \kappa} + 2 N t h_t \sup_{ x \in B_{ h_t ( a) } } \pi_1^{f_0} ( x) ,$$
which implies finally \eqref{eq:finalouf} by Theorem \ref{theo:invmeasure}. \hfill $\qed$

\section*{Acknowledgments}
We thank an anonymous referee for helpful comments and suggestions. This research has been conducted as part of the project Labex MME-DII (ANR11-LBX-0023-01),  as part of the Agence Nationale de la Recherche
PIECE 12-JS01-0006-01 and as part of the activities of FAPESP Research,
Dissemination and Innovation Center for Neuromathematics (grant
2013/07699-0, S.\ Paulo Research Foundation).

\end{document}